\documentclass[12pt]{article}

\usepackage[margin=1.25in]{geometry}

\usepackage{amssymb,amsmath}
\usepackage{amsthm}
\usepackage{hyperref}
\usepackage{mathrsfs}
\usepackage{textcomp}
\usepackage{enumerate}
\usepackage{framed}
\usepackage{ulem}
\usepackage{color,xcolor}

\usepackage{soul}
\soulregister\cite7 
\soulregister\citep7 
\soulregister\citet7 
\soulregister\ref7 
\soulregister\pageref7 
\soulregister\eqref7

\hypersetup{
    colorlinks,%
    citecolor=blue,%
    filecolor=blue,%
    linkcolor=blue,%
    urlcolor=black
}

\newtheorem{theorem}{Theorem}[section]

\newtheorem{lemma}{Lemma}[section]
\newtheorem{remark}{Remark}[section]
\newtheorem{proposition}{Proposition}[section]

\newtheorem*{thm}{Theorem A}
\newtheorem*{thmB}{Theorem B}

\numberwithin{equation}{section}

\makeatletter

\newdimen\bibspace
\renewenvironment{thebibliography}[1]{%
 \section*{\refname 
       \@mkboth{\MakeUppercase\refname}{\MakeUppercase\refname}}%
     \list{\@biblabel{\@arabic\c@enumiv}}%
          {\settowidth\labelwidth{\@biblabel{#1}}%
           \leftmargin\labelwidth
           \advance\leftmargin\labelsep
           \itemsep\bibspace
           \parsep\z@skip     %
           \@openbib@code
           \usecounter{enumiv}%
           \let\p@enumiv\@empty
           \renewcommand\theenumiv{\@arabic\c@enumiv}}%
     \sloppy\clubpenalty4000\widowpenalty4000%
     \sfcode`\.\@m}
    {\def\@noitemerr
      {\@latex@warning{Empty `thebibliography' environment}}%
     \endlist}

\makeatother

           \newcommand{\ud}{\mathrm{d}}
\newcommand{\be}{\begin{equation}}      \newcommand{\ee}{\end{equation}}

\begin{document}

\title{A weighted Sobolev-Poincar\'e type trace inequality on Riemannian manifolds}

\author{Zhongwei Tang\thanks{Z. Tang is supported by National Natural Science Foundation of China (12071036).},\, Ning Zhou}

\date{}

\maketitle

\begin{abstract}
Given $(M, g)$ a smooth compact $(n+1)$-dimensional Riemannian manifold with boundary $\partial M$. Let $\rho$ be a defining function of $M$ and $\sigma \in(0,1)$. In this paper we study a weighted Sobolev-Poincar\'e type trace inequality corresponding to the embedding of $W^{1,2}(\rho^{1-2 \sigma}, M) \hookrightarrow L^{p}(\partial M)$, where $p=\frac{2 n}{n-2 \sigma}$. More precisely, under some assumptions on the manifold, we prove that there exists a constant $B>0$ such that, for all $u \in W^{1,2}(\rho^{1-2\sigma}, M)$,
$$
\Big(\int_{\partial M}|u|^{p} \,\ud s_{g}\Big)^{2/p} \leq \mu^{-1} \int_{M} \rho^{1-2 \sigma}|\nabla_{g} u|^{2}  \,\ud  v_{g}+B \Big|\int_{\partial M} |u|^{p-2}u  \,\ud s_{g}\Big|^{2/(p-1)}.
$$
This inequality is sharp in the sense that $\mu^{-1}$ cannot be replaced by any smaller constant. Moreover, unlike the classical Sobolev inequality, $\mu^{-1}$ does not depend on $n$ and $\sigma$ only, but depends on the manifold.
\end{abstract}


\section{Introduction}

Let $(M, g)$ be a $(n+1)$-dimensional smooth compact Riemannian manifold with boundary $\partial M$, $n \geq 2$. There has been much work on the sharp Sobolev-type inequalities and sharp Sobolev-type trace inequalities on $(M,g)$ and their applications, see, for example, \cite{LiZhuSharp1997}, \cite{LiZhuSharp1998}, \cite{JinXiongSharp2013}, \cite{JinXiongA2015}, \cite{TangXiongZhou2021}, and the references therein.

In \cite{LiZhuSharp1997}, Li and Zhu established the sharp Sobolev trace inequalities corresponding to the embedding of $H^1(M) \hookrightarrow L^{\frac{2n}{n-1}}(\partial M)$, they proved that there exists a constant $A_1=A_1(M,g)>0$ such that, for any $u\in H^1(M)$,
\be\label{eq:LiZhu1997}
\Big(\int_{\partial M}|u|^{\frac{2n}{n-1}} \,\ud s_{g}\Big)^{\frac{n-1}{n}} \leq S(n) \int_{M}|\nabla_{g} u|^{2} \,\ud v_{g}+A_1 \int_{\partial M} u^{2} \,\ud s_{g},
\ee
where $S(n)=\frac{2}{n-1}\omega_{n}^{-{1}/{n}}$ is the best constant (see \cite{EscobarSharp1988} or \cite{BecknerSharp1993}), $\omega_{n}$ denotes the volume of the unit sphere $\mathbb{S}^n$ of $\mathbb{R}^{n+1}$, $\ud v_{g}$ is the volume form of $(M, g)$ and $\ud s_{g}$ is the induced volume form on $\partial M$.

Later, Holcman and Humbert \cite{HolcmanHumbertA2001} studied the following Sobolev-Poincar\'e type inequality:
\be\label{eq:HolcmanHumbert2001}
\Big(\int_{\partial M}|u|^{\frac{2n}{n-1}} \,\ud s_{g}\Big)^{\frac{n-1}{n}} \leq A \int_{M}|\nabla_g u|^{2}\,\ud v_g+B\Big|\int_{\partial M}|u|^{\frac{2}{n-1}} u\,\ud s_g\Big|^{\frac{2(n-1)}{n+1}}.
\ee
Contrary to \eqref{eq:LiZhu1997}, they showed that the best constant in \eqref{eq:HolcmanHumbert2001} does not depend on the dimension only, but depends on the geometry of the boundary.

Considerable effort has also been devoted to the study of Sobolev-Poincar\'e type inequalities on compact Riemannian manifolds without boundary, we refer to \cite{DruetHebeyVaugonSharp2001, HebeySharp2002, DruetHebeyAsymptotics2002, ZhuOn2004}, among others.

Let $\mathbb{R}_{+}^{n+1}=\{(x, t) \mid x \in \mathbb{R}^{n},\, t>0\}$, $\sigma\in (0,1)$ and $W^{1,2}(t^{1-2\sigma}, \mathbb{R}_{+}^{n+1})$ be the homogeneous weighted Sobolev space with weight $t^{1-2 \sigma}$, i.e., the closure of $C_{c}^{\infty}(\overline{\mathbb{R}}_{+}^{n+1})$ under the norm
$$
\|U\|_{W^{1,2}(t^{1-2 \sigma}, \mathbb{R}_{+}^{n+1})}=\Big(\int_{\mathbb{R}_{+}^{n+1}} t^{1-2 \sigma}(|U|^{2}+|\nabla U|^{2}) \,\ud x\,\ud t\Big)^{1/2}.
$$
The sharp weighted Sobolev trace inequality asserts that (see \cite{LiebSharp1983, CotsiolisTavoularisBest2004})
\be\label{SharpWSTineRn}
\|U(\cdot, 0)\|_{L^{\frac{2 n}{n-2 \sigma}}(\mathbb{R}^{n})}^2 \leq S(n, \sigma) \int_{\mathbb{R}_{+}^{n+1}} t^{1-2 \sigma}|\nabla U(x, t)|^{2} \,\ud x\,\ud t
\ee
for all $U\in W^{1,2}(t^{1-2 \sigma}, \mathbb{R}_{+}^{n+1})$, where
\be\label{eq:SNSigmadefi}
S(n,\sigma)=\frac{1}{2\pi^{\sigma}}\frac{\Gamma(\sigma)}{\Gamma(1-\sigma)}\frac{\Gamma((n-2 \sigma) / 2)}{\Gamma((n+2 \sigma) / 2)}\Big(\frac{\Gamma(n)}{\Gamma(n / 2)}\Big)^{{2 \sigma}/{n}}.
\ee

A function $\rho \in C^{\infty}(\overline{M})$ is called a defining function of $M$ if
$$
\rho>0\,\,\, \text { in }\, M, \quad \rho=0\,\,\, \text { on }\, \partial M, \quad \nabla_{g} \rho \neq 0\,\,\, \text { on }\, \partial M.
$$
The weighted Sobolev space $W^{1,2}(\rho^{1-2\sigma}, M)$ is defined as the closure of $C^{\infty}(\overline{M})$ under the norm
\be\label{eq:W12norm}
\|U\|_{W^{1,2}(\rho^{1-2 \sigma}, M)}=\Big(\int_{M} \rho^{1-2 \sigma}|\nabla_{g} U|^{2} \,\ud v_{g}+\int_{\partial M} U^{2} \,\ud s_g\Big)^{{1}/{2}}.
\ee
The weighted Sobolev space $W^{1,2}(t^{1-2 \sigma}, \mathbb{R}_{+}^{n+1})$ and $W^{1,2}(\rho^{1-2\sigma}, M)$ play important roles in the study of the fractional Nirenberg problem and the fractional Yamabe problem, respectively, see \cite{JinLiXiongOn2014, JinLiXiongOn2015, GonzalezQing2013, GonzalezWang2018, KimMussoWei2017, KimMussoWei2018, KimMussoWeiA2021} and references therein.

In \cite{JinXiongSharp2013}, Jin and Xiong established the sharp weighted Sobolev trace inequalities of type \eqref{SharpWSTineRn} on Riemannian manifolds with boundaries:

\begin{thm}
For $n \geq 2$, let $(M, g)$ be a $(n+1)$-dimensional smooth compact Riemannian manifold with boundary $\partial M$. Let $\sigma \in(0, 1/2]$, and $\rho$ be a defining function of $M$ satisfying $|\nabla_{g} \rho|=1$ on $\partial M$. Then there exists a constant ${A}_2={A}_2(M, g, \sigma, \rho)>0$ such that, for all $u \in W^{1,2}(\rho^{1-2 \sigma}, M)$,
\be\label{JX-ineq1}
\Big(\int_{\partial M}|u|^{\frac{2n}{n-2\sigma}} \,\ud s_{g}\Big)^{\frac{n-2\sigma}{n}} \leq S(n, \sigma) \int_{M} \rho^{1-2 \sigma}|\nabla_{g} u|^{2}  \,\ud  v_{g}+{A}_2 \int_{\partial M} u^{2}  \,\ud s_{g}.
\ee

If $\sigma \in(1/2, 1)$, $n \geq 4$ and $\partial M$ is totally geodesic. Let $\rho$ be a defining function of $M$ satisfying $\rho(x)=d(x)+O(d(x)^{3})$ as $d(x) \rightarrow 0$, where $d(x)$ denotes the distance between $x$ and $\partial M$ with respect to the metric $g$. Then \eqref{JX-ineq1} still holds for all $u \in W^{1,2}(\rho^{1-2 \sigma}, M)$.
\end{thm}

Note that if $\sigma \in(1/2, 1)$, additional assumptions about the manifold are required in order to obtain sharp inequality \eqref{JX-ineq1}. However, the following conclusion holds for any $\sigma\in(0,1)$ without additional assumptions, see \cite[Proposition 2.5]{JinXiongSharp2013}.

\begin{thmB}
For $n \geq 2$, let $(M, g)$ be a $(n+1)$-dimensional smooth compact Riemannian manifold with boundary $\partial M$. Let $\sigma \in(0, 1)$, and $\rho$ be a defining function of $M$ satisfying $|\nabla_{g} \rho|=1$ on $\partial M$. Then for any $\varepsilon>0$, there exists a constant $A_{\varepsilon}>0$ such that, for all $u \in W^{1,2}(\rho^{1-2 \sigma}, M)$,
$$
\Big(\int_{\partial M}|u|^{\frac{2n}{n-2\sigma}} \,\ud s_{g}\Big)^{\frac{n-2\sigma}{n}} \leq (S(n, \sigma)+\varepsilon) \int_{M} \rho^{1-2 \sigma}|\nabla_{g} u|^{2}  \,\ud  v_{g}+A_{\varepsilon} \int_{\partial M} u^{2}  \,\ud s_{g}.
$$
\end{thmB}

From now on we denote $p=2 n /(n-2 \sigma)$ for simplicity.

In this paper, we study a weighted Sobolev-Poincar\'e type trace inequality obtained by replacing the $L^2$ remainder term in \eqref{JX-ineq1} by another nonlinear term.

Before  stating our main results, let us first denote
\be\label{eq:mudefinition}
\mu:=\inf _{u \in \Lambda} \frac{\int_{M}\rho^{1-2\sigma}|\nabla_{g} u|^{2} \,\ud v_{g}}{(\int_{\partial M}|u|^{p}\,\ud s_{g})^{2/p}},
\ee
where
$$
\Lambda:=\Big\{u \in W^{1,2}(\rho^{1-2\sigma}, M) \mid \int_{\partial M}|u|^{p-2} u\,\ud s_{g}=0,\, u\not\equiv 0\, \text{ on }\, \partial M\Big\}.
$$
Our main result is as follows:
\begin{theorem}\label{thm}
Let $(M, g)$ be a $(n+1)$-dimensional smooth compact Riemannian manifold with boundary $\partial M$, $n\geq 2$. Let $\sigma \in(0, 1)$, and $\rho$ be a defining function of $M$ with $|\nabla_{g} \rho|=1$ on $\partial M$.  Then there exists a constant $B>0$ such that, for all $u \in W^{1,2}(\rho^{1-2\sigma}, M)$,
\be\label{eq:maininequality}
\Big(\int_{\partial M}|u|^{p} \,\ud s_{g}\Big)^{2/p} \leq \mu^{-1} \int_{M} \rho^{1-2 \sigma}|\nabla_{g} u|^{2}  \,\ud  v_{g}+B \Big|\int_{\partial M} |u|^{p-2}u  \,\ud s_{g}\Big|^{2/(p-1)},
\ee  under the following conditions satisfied:
\begin{enumerate}[(1)]
  \item
$n=3$, $\sigma\in (0,1/2)$, and $H(P) > 0$ at some point $P \in \partial M$,
  \item
or $n \geq 4$, $\sigma\in (0, 1)$, and $H(P) > 0$ at some point $P \in \partial M$,
  \item
or $n \geq 5$, $\sigma\in (0, 1)$, $H(P)=0$ and
$$
\frac{3n^2-6n-4\sigma^2+4}{12(1-\sigma)(n-1)(n-2-2\sigma)}\bar{R}(P)+\|\pi\|^2(P)+
\frac{3n-2-2\sigma}{3n-6-6\sigma}R_{tt}(P)>0
$$
at some point $P \in \partial M$,
\end{enumerate}
where $H$ is the mean curvature of $\partial M$, $\bar{R}$ denotes the scalar curvature for the metric induced by $g$ on $\partial M$, $\pi$ is the second fundamental form of $\partial M$, and $R_{tt}$ is the component of the Ricci curvature tensor in $M$.

\end{theorem}

In fact, under the conditions of Theorem \ref{thm}, we can prove that $\mu<S(n,\sigma)^{-1}$, where $S(n,\sigma)$ is defined in \eqref{eq:SNSigmadefi}. Then, by a standard variational arguments, we have the following existence results. 

\begin{theorem}\label{thm:2}
Under the conditions of Theorem \ref{thm},  there exists a function $u\in \Lambda$ such that
$$
\mu=\frac{\int_{M}\rho^{1-2\sigma}|\nabla_g u|^{2}\,\ud v_g}{(\int_{\partial M}|u|^{p}\,\ud s_g)^{2/p}}\quad \text { and }\quad \int_{\partial M}|u|^{p}\,\ud s_g=1.
$$
Moreover, $u$ satisfies
$$
\begin{cases}
\operatorname{div}_g(\rho^{1-2 \sigma} \nabla_g u)=0 & \text { in }\, M, \\
-\lim_{\rho \rightarrow 0^+} \rho^{1-2 \sigma} \partial_{\rho} u=\mu |u|^{p-2}u & \text { on }\, \partial M.
\end{cases}
$$
\end{theorem}

The proof of Theorem \ref{thm:2} is standard, see for example \cite{CherrierProbl1984, LiZhuSharp1997, LiZhuSharp1998, JinXiongSharp2013}, so we omit it.

\begin{remark}
Recall that the best constant in the sharp weighted Sobolev trace inequality \eqref{JX-ineq1} is $S(n,\sigma)$. In \eqref{eq:maininequality}, the best constant $\mu^{-1}>S(n,\sigma)$. On the other hand, if $n>6\sigma$, the exponent of the nonlinear term in the right hand side of \eqref{eq:maininequality} is less than 2.
\end{remark}

The rest of this paper is organized as follows. In Section 2, we prove some preliminary results related to the main inequality \eqref{eq:maininequality}. In Section 3, we prove Theorem \ref{thm} by constructing some test functions. In Appendix, we give some estimates required in the proof of Theorem \ref{thm}.

\section{Some preliminary results}
In this section, we present some preliminary results. For all $u \in W^{1,2}(\rho^{1-2 \sigma}, M)$ and $A, B > 0$, consider the following inequality
\be\label{SP-ineq}
\Big(\int_{\partial M}|u|^p \,\ud s_{g}\Big)^{2/p} \leq A \int_{M} \rho^{1-2 \sigma}|\nabla_{g} u|^{2}  \,\ud  v_{g}+B \Big|\int_{\partial M} |u|^{p-2}u  \,\ud s_{g}\Big|^{2/(p-1)}.
\ee
Define
$$
A_{0}=\inf \{A>0 \mid \exists\, B>0 \text { s.t. } \eqref{SP-ineq} \text { is valid}\},
$$
where by ``\eqref{SP-ineq} is valid'' we mean that \eqref{SP-ineq} holds with $A$ and $B$, for all $u \in W^{1,2}(\rho^{1-2 \sigma}, M)$. By definition, $A_0$ is the best constant for \eqref{SP-ineq} in the sense that \eqref{SP-ineq} does not hold with some $A^{\prime}<A_0$ in place of $A_0$.

Our first result in this section is about the validity of \eqref{SP-ineq}, which is:

\begin{proposition}\label{pro:1}
Let $(M, g)$ be a $(n+1)$-dimensional smooth compact Riemannian manifold with boundary $\partial M$, $n\geq 2$. Let $\sigma \in(0, 1)$, and $\rho$ be a defining function of $M$ with $|\nabla_{g} \rho|=1$ on $\partial M$. Then there exists $A, B>0$ such that \eqref{SP-ineq} is valid on $M$.
\end{proposition}

Our second aim in this section is to study the best constant $A_0$ associated with inequality \eqref{SP-ineq}. We are able to prove that $A_{0}$ does not depend only on $n$ and $\sigma$, but on the geometry of the manifold. In fact we have the following result.

\begin{proposition}\label{pro:2}
Assume as in Proposition \ref{pro:1}. Let $\mu$ is defined as in \eqref{eq:mudefinition}, namely
$$
\mu=\inf _{u \in \Lambda} \frac{\int_{M}\rho^{1-2\sigma}|\nabla_{g} u|^{2} \,\ud v_{g}}{(\int_{\partial M}|u|^{p}\,\ud s_{g})^{2/p}},
$$
where
$$
\Lambda=\Big\{u \in W^{1,2}(\rho^{1-2\sigma}, M) \mid \int_{\partial M}|u|^{p-2} u\,\ud s_{g}=0,\, u\not\equiv 0\, \text{ on }\, \partial M\Big\}.
$$
Then $A_{0}=\mu^{-1}$.
\end{proposition}
Now a natural question is that in what conditions, the infimum $A_0=\mu^{-1}$  can be attained,  namely the inequality \eqref{eq:maininequality} holds. In fact we have the following proposition.

\begin{proposition}\label{pro:3}
Assume as in Proposition \ref{pro:1}.  If $\mu<S(n,\sigma)^{-1}$, where $S(n,\sigma)$ is defined in \eqref{eq:SNSigmadefi}.  Then there exists a constant $B>0$ such that, for all $u \in W^{1,2}(\rho^{1-2\sigma}, M)$, the inequality \eqref{eq:maininequality} holds.
\end{proposition}

\begin{remark}
By  Proposition \ref{pro:3},  to prove our main result  Theorem \ref{thm}, we only need to show that under the assumptions  about the mean curvature on the boundary $\partial M$ of  the manifold $M$ in Theorem \ref{thm},  it holds that  $\mu<S(n,\sigma)^{-1}$  and this will be done in the following Section 3.
\end{remark}

Now,   we are going to  prove Propositions \ref{pro:1}, \ref{pro:2}, and \ref{pro:3}.

\begin{proof}[Proof of Proposition \ref{pro:1}]
We begin to prove Proposition \ref{pro:1} through an argument by contradiction. Suppose the contrary of Proposition \ref{pro:1} is true, then we have, for all $\alpha\geq 1$, there exists $u_{\alpha} \in W^{1,2}(\rho^{1-2 \sigma}, M)$ such that
\be\label{eq:thm1-1}
\Big(\int_{\partial M}|u_{\alpha}|^{p} \,\ud s_{g}\Big)^{2/p} > \alpha \Big(\int_{M} \rho^{1-2 \sigma}|\nabla_{g} u_{\alpha}|^{2}  \,\ud  v_{g}+\Big|\int_{\partial M} |u_{\alpha}|^{p-2}u_{\alpha}  \,\ud s_{g}\Big|^{2/(p-1)}\Big).
\ee
By homogeneity we may assume that
\be\label{eq:thm1-2}
\int_{\partial M}|u_{\alpha}|^{p}\,\ud s_g=1.
\ee

Firstly, it follows from \eqref{eq:thm1-2} and H\"older inequality that $\{u_{\alpha}\}$ is bounded in $L^{2}(\partial M)$. Recall the definition of the norm $\|\cdot\|_{W^{1,2}(\rho^{1-2 \sigma},M)}$ given in \eqref{eq:W12norm}, $\{u_{\alpha}\}$ is bounded in $W^{1,2}(\rho^{1-2 \sigma},M)$. Thus, by choosing a subsequence if necessary, there exists $u \in W^{1,2}(\rho^{1-2 \sigma},M)$ such that as $\alpha \rightarrow \infty$,
\be\label{eq:thm1-3}
\begin{aligned}
&u_{\alpha} \rightharpoonup u \quad \text { weakly in }\, W^{1,2}(\rho^{1-2 \sigma}, M),\\
&u_{\alpha} \to u \quad \text { strongly in }\, L^{2}(\partial M).
\end{aligned}
\ee

By weakly lower semi-continuous of the norm, we have
$$
\int_{M}\rho^{1-2\sigma}|\nabla_g u|^{2}\,\ud v_g \leq \liminf _{\alpha \rightarrow \infty} \int_{M}\rho^{1-2\sigma}|\nabla_g u_{\alpha}|^{2}\,\ud v_g.
$$
On the other hand, by \eqref{eq:thm1-1} and \eqref{eq:thm1-2},
$$
\liminf _{\alpha \rightarrow \infty} \int_{M}\rho^{1-2\sigma}|\nabla_g u_{\alpha}|^{2}\,\ud v_g=0.
$$
Therefore, $u$ is constant. This together with \eqref{eq:thm1-3} implies that
$$
\lim _{\alpha \rightarrow \infty}\|u_{\alpha}\|_{W^{1,2}(\rho^{1-2 \sigma},M)}=\|u\|_{W^{1,2}(\rho^{1-2 \sigma},M)}.
$$
Furthermore, $u_{\alpha}\to u$ strongly in $W^{1,2}(\rho^{1-2 \sigma},M)$ as $\alpha \rightarrow \infty$. By Sobolev embeddings theorem, \eqref{eq:thm1-1} and \eqref{eq:thm1-2}, we obtain that
$$
\int_{\partial M}|u|^{p-2} u\,\ud s_g=0\quad \text{ and }\quad \int_{\partial M}|u|^{p}\,\ud s_g=1,
$$
which contradicts to the fact that $u$ is constant. This finishes the proof of Proposition \ref{pro:1}.
\end{proof}

\begin{proof}[Proof of Proposition \ref{pro:2}]
Choosing $u \in \Lambda$, i.e., $u\in W^{1,2}(\rho^{1-2\sigma}, M)$ satisfies
$$
\int_{\partial M}|u|^{p-2} u\,\ud s_g=0\quad \text{ and }\quad u\not\equiv 0\,\,\, \text{ on }\, \partial M.
$$
For all $\varepsilon>0$, by the definition of best constant $A_{0}$, we have
$$
\Big(\int_{\partial M}|u|^{p}\,\ud s_g\Big)^{2/p} \leq(A_{0}+\varepsilon) \int_{M}\rho^{1-2\sigma}|\nabla_g u|^{2}\,\ud v_g.
$$
Then one finds from the definition of $\mu$ that $A_{0}\geq \mu^{-1}$. It remains to prove that $A_{0}\leq \mu^{-1}$. We argue by contradiction and assume that $\mu^{-1}<A_{0}$. Choosing $S \in (\mu^{-1}, A_{0})$. For every $\alpha\geq1$, define the functional
$$
I_{\alpha}(u):=\frac{\int_{M}\rho^{1-2\sigma}|\nabla_g u|^{2}\,\ud v_g+\alpha|\int_{\partial M}|u|^{p-2} u\,\ud s_g|^{2/(p-1)}}{(\int_{\partial M}|u|^{p}\,\ud s_g)^{2/p}},
$$
where $u \in W^{1,2}(\rho^{1-2 \sigma}, M)$ and $u|_{\partial M} \not \equiv 0$. Define
\be\label{eq:defxi}
\xi_{\alpha}:=\inf _{u \in W^{1,2}(\rho^{1-2 \sigma}, M),\, u|_{\partial M} \not \equiv 0} I_{\alpha}(u).
\ee

Note that the proof of Theorem \ref{thm} implies that $\mu \leq S(n,\sigma)^{-1}$ (See Section 3).  Since $S<A_{0}$, by the definition of best constant $A_{0}$, for any $\alpha\geq 1$, there exists a function $v_{\alpha} \in W^{1,2}(\rho^{1-2 \sigma}, M)$ satisfies
$$
\Big(\int_{\partial M}|v_{\alpha}|^{p} \,\ud s_{g}\Big)^{2/p} > S \int_{M} \rho^{1-2 \sigma}|\nabla_{g} v_{\alpha}|^{2}  \,\ud  v_{g}+\alpha S \Big|\int_{\partial M} |v_{\alpha}|^{p-2}v_{\alpha}  \,\ud s_{g}\Big|^{2/(p-1)}.
$$
This implies that $\xi_{\alpha}<S^{-1}$. Since $S>\mu^{-1} \geq S(n,\sigma)$, we have $\xi_{\alpha}<S(n,\sigma)^{-1}$. By standard calculus of variations (see for example \cite[Proposition 1.2]{LiZhuSharp1997}, \cite[Proposition 2.4]{JinXiongSharp2013}), there exists a function $u_{\alpha} \in W^{1,2}(\rho^{1-2 \sigma}, M)$ satisfying
$$
\xi_{\alpha}=\int_{M}\rho^{1-2\sigma}|\nabla_g u_{\alpha}|^{2}\,\ud v_g+\alpha\Big|\int_{\partial M}|u_{\alpha}|^{p-2} u_{\alpha}\,\ud s_g\Big|^{2/(p-1)}
$$
and
\be\label{eq:thm2-1}
\int_{\partial M}|u_{\alpha}|^{p}\,\ud s_g=1.
\ee

Let
$$
C_{\alpha}:=\Big|\int_{\partial M}|u_{\alpha}|^{p-2} u_{\alpha}\,\ud s_g\Big|.
$$
We claim that $C_{\alpha} \neq 0$. Indeed, if $C_{\alpha}=0$, then $u_{\alpha}\in \Lambda$. Thus, by the definition of $\mu$, we have
$$
\xi_{\alpha}=\frac{\int_{M}\rho^{1-2\sigma}|\nabla_{g} u_{\alpha}|^{2} \,\ud v_{g}}{(\int_{\partial M}|u_{\alpha}|^{p}\,\ud s_{g})^{2/p}}\geq \inf _{u \in \Lambda} \frac{\int_{M}\rho^{1-2\sigma}|\nabla_{g} u|^{2} \,\ud v_{g}}{(\int_{\partial M}|u|^{p}\,\ud s_{g})^{2/p}}=\mu,
$$
which contradicts to the fact that $\xi_{\alpha}<S^{-1}<\mu$. Hence, $u_{\alpha}$ satisfies the Euler-Lagrange equation
$$
\begin{cases}
\operatorname{div}_{g}(\rho^{1-2 \sigma} \nabla_{g} u_{\alpha})=0 & \text{ in }\, M,\\
-\lim _{\rho \rightarrow 0^+} \rho^{1-2 \sigma} \partial_{\rho} u_{\alpha}=\xi_{\alpha}|u_{\alpha}|^{p-2}u_{\alpha}-\alpha C_{\alpha}^{\frac{3-p}{p-1}}|u_{\alpha}|^{p-2} & \text{ on }\, \partial M.
\end{cases}
$$

Similarly to the proof of Proposition \ref{pro:1}, $\{u_{\alpha}\}$ is bounded in $W^{1,2}(\rho^{1-2 \sigma},M)$. Therefore, there exists $u \in W^{1,2}(\rho^{1-2 \sigma},M)$ such that as $\alpha\to \infty$,
$$
\begin{aligned}
&u_{\alpha} \rightharpoonup u \quad \text { weakly in }\, W^{1,2}(\rho^{1-2 \sigma}, M),\\
&u_{\alpha} \to u \quad \text { strongly in }\, L^{q}(\partial M),
\end{aligned}
$$
where $q<p$. It follows from Br\'ezis-Lieb lemma that $u_{\alpha}$ and $u$ satisfy
\be\label{eq:thm2-3}
\int_{\partial M} |u_{\alpha}|^p \,\ud s_{g}-\int_{\partial M}|u_{\alpha}-u|^p \,\ud s_{g}-\int_{\partial M} |u|^p \,\ud s_{g} \rightarrow 0 \quad \text { as }\, \alpha \rightarrow \infty,
\ee
and, in view of \eqref{eq:thm2-1},
\be\label{eq:thm2-4}
\int_{\partial M}|u_{\alpha}-u|^p \,\ud s_{g} \leq 1+o(1), \quad \int_{\partial M} |u|^p \,\ud s_{g} \leq 1,
\ee
where $o(1) \rightarrow 0$ as $\alpha \rightarrow \infty$.

Since $\xi_{\alpha}=I_{\alpha}(u_{\alpha})<S^{-1}$, it follows from the definition of $I_{\alpha}$ that
\be\label{eq:thm2-5}
\Big|\int_{\partial M}|u|^{p-2} u\,\ud s_g\Big|=\lim _{\alpha \rightarrow \infty} \Big|\int_{\partial M}|u_{\alpha}|^{p-2} u_{\alpha}\,\ud s_g\Big|=0.
\ee
We claim that $u \not \equiv 0$. Indeed, by Theorem B, for any $\varepsilon>0$, there exists a constant ${A}_{\varepsilon}>0$ such that, for any $\alpha\geq 1$,
$$
1=\Big(\int_{\partial M}|u_{\alpha}|^{p}\,\ud s_g\Big)^{2/p} \leq (S(n,\sigma)+\varepsilon) \int_{M}\rho^{1-2\sigma}|\nabla_g u_{\alpha}|^{2}\,\ud v_g+{A}_{\varepsilon} \int_{\partial M} u_{\alpha}^{2}\,\ud s_g.
$$
Since $\xi_{\alpha}=I_{\alpha}(u_{\alpha})<S^{-1}$, using the definition of $I_{\alpha}$, we have
$$
\limsup _{\alpha \rightarrow \infty} \int_{M}\rho^{1-2\sigma}|\nabla_g u_{\alpha}|^{2}\,\ud v_g \leq S^{-1}.
$$
In addition, since $S>\mu^{-1} \geq S(n,\sigma)$, choosing $\varepsilon>0$ small enough, we get
$$
0<1-(S(n,\sigma)+\varepsilon) S^{-1} \leq {A}_{\varepsilon} \int_{\partial M} u^{2}\,\ud s_g,
$$
which concludes the proof of the claim.

By the compact embedding of $W^{1,2}(\rho^{1-2 \sigma}, M)$ to $L^{2}(\partial M)$, Theorem B, \eqref{eq:thm2-5}, \eqref{eq:defxi}, \eqref{eq:thm2-4}, \eqref{eq:thm2-3}, and \eqref{eq:thm2-1}, we have
$$
\begin{aligned}
\xi_{\alpha}=&\int_{M}\rho^{1-2\sigma}|\nabla_g u_{\alpha}|^{2}\,\ud v_g+\alpha\Big|\int_{\partial M}|u_{\alpha}|^{p-2} u_{\alpha}\,\ud s_g\Big|^{2/(p-1)}\\
\geq&\int_{M}\rho^{1-2\sigma}|\nabla_g (u_{\alpha}-u)|^{2}\,\ud v_g+\int_{M}\rho^{1-2\sigma}|\nabla_g u|^{2}\,\ud v_g\\
=&\int_{M}\rho^{1-2\sigma}|\nabla_g (u_{\alpha}-u)|^{2}\,\ud v_g+\frac{{A}_{\varepsilon}}{S(n,\sigma)+\varepsilon}\int_{\partial M} |u_{\alpha}-u|^2\,\ud s_g+\int_{M}\rho^{1-2\sigma}|\nabla_g u|^{2}\,\ud v_g+o(1)\\
\geq&\frac{1}{S(n,\sigma)+\varepsilon}\Big(\int_{\partial M}|u_{\alpha}-u|^{p}\,\ud s_g\Big)^{2/p}+\int_{M}\rho^{1-2\sigma}|\nabla_g u|^{2}\,\ud v_g+o(1)\\
=&\frac{1}{S(n,\sigma)+\varepsilon}\Big(\int_{\partial M}|u_{\alpha}-u|^{p}\,\ud s_g\Big)^{2/p}+\int_{M}\rho^{1-2\sigma}|\nabla_g u|^{2}\,\ud v_g+\alpha\Big|\int_{\partial M}|u|^{p-2} u\,\ud s_g\Big|^{2/(p-1)}+o(1)\\
\geq&\frac{1}{S(n,\sigma)+\varepsilon}\Big(\int_{\partial M}|u_{\alpha}-u|^{p}\,\ud s_g\Big)^{2/p}+\xi_{\alpha}\Big(\int_{\partial M}|u|^p\,\ud s_g\Big)^{2/p}+o(1)\\
\geq&\frac{1}{S(n,\sigma)+\varepsilon}\int_{\partial M}|u_{\alpha}-u|^{p}\,\ud s_g+\xi_{\alpha}\int_{\partial M}|u|^p\,\ud s_g+o(1)\\
=&\Big(\frac{1}{S(n,\sigma)+\varepsilon}-\xi_{\alpha}\Big)\int_{\partial M}|u_{\alpha}-u|^{p}\,\ud s_g+\xi_{\alpha}+o(1).
\end{aligned}
$$
Since $\xi_{\alpha}<S^{-1}<S(n,\sigma)^{-1}$, choosing $\varepsilon>0$ small enough, we can derive that $\|u_{\alpha}-u\|_{L^p(\partial M)}\to 0$ as $\alpha \rightarrow \infty$. In particular,
$$
\lim_{\alpha\to \infty}\int_{\partial M}|u_{\alpha}|^{p}\,\ud s_g=\int_{\partial M}|u|^{p}\,\ud s_g.
$$
By weakly lower semi-continuous of the norm, we have
$$
\int_{M}\rho^{1-2\sigma}|\nabla_g u|^{2}\,\ud v_g\leq \liminf _{\alpha \rightarrow \infty}\int_{M}\rho^{1-2\sigma}|\nabla_g u_{\alpha}|^{2}\,\ud v_g.
$$
Therefore,
$$
\frac{\int_{M}\rho^{1-2\sigma}|\nabla_g u|^{2}\,\ud v_g}{(\int_{\partial M}|u|^{p}\,\ud s_g)^{2/p}}\leq\liminf _{\alpha \rightarrow \infty} \frac{\int_{M}\rho^{1-2\sigma}|\nabla_g u_{\alpha}|^{2}\,\ud v_g}{(\int_{\partial M}|u_{\alpha}|^{p}\,\ud s_g)^{2/p}} \leq \liminf _{\alpha \rightarrow \infty} \xi_{\alpha} \leq S^{-1}<\mu.
$$
Together with \eqref{eq:thm2-5}, this gives a contradiction to the definition of $\mu$, and hence completes the proof of Proposition \ref{pro:2}.
\end{proof}

\begin{proof}[Proof of Proposition \ref{pro:3}]
We prove this proposition using an argument by contradiction. Suppose the contrary of Proposition \ref{pro:3} is true, then we have, for all $\alpha\geq 1$,
$$
\xi_{\alpha}<\mu,
$$
where $\xi_{\alpha}$ is defined in \eqref{eq:defxi}. Since $\mu<S(n,\sigma)^{-1}$, we have $\xi_{\alpha}<S(n,\sigma)^{-1}$. As in the proof of Proposition \ref{pro:2}, there exists $u_{\alpha}\in W^{1,2}(\rho^{1-2 \sigma},M)$ satisfying
$$
\xi_{\alpha}=I_{\alpha}(u_{\alpha})\quad\text{ and }\quad \int_{\partial M}|u_{\alpha}|^{p}\,\ud s_g=1.
$$
It follows that $u_{\alpha}$ satisfies
\be\label{eq:thm3-1}
\begin{cases}
\operatorname{div}_{g}(\rho^{1-2 \sigma} \nabla_{g} u_{\alpha})=0 & \text{ in }\, M,\\
-\lim _{\rho \rightarrow 0^+} \rho^{1-2 \sigma} \partial_{\rho} u_{\alpha}=\xi_{\alpha}|u_{\alpha}|^{p-2}u_{\alpha}-\alpha C_{\alpha}^{\frac{3-p}{p-1}}|u_{\alpha}|^{p-2} & \text{ on }\, \partial M,
\end{cases}
\ee
where
$$
C_{\alpha}=\Big|\int_{\partial M}|u_{\alpha}|^{p-2} u_{\alpha}\,\ud s_g\Big|
$$
satisfying $\lim_{\alpha\to \infty} C_{\alpha}=0$.

Again proceeding as in the proof of Proposition \ref{pro:2}, there exists $u \in W^{1,2}(\rho^{1-2\sigma}, M)$ such that as $\alpha\to \infty$,
$$
\begin{aligned}
&u_{\alpha} \rightharpoonup u \quad \text { weakly in }\, W^{1,2}(\rho^{1-2 \sigma}, M),\\
&u_{\alpha} \to u \quad \text { strongly in }\, L^{q}(\partial M),
\end{aligned}
$$
where $q<p$. Using the fact that $\xi_{\alpha}<\mu<S(n,\sigma)^{-1}$, similar to the proof of Proposition \ref{pro:2}, we conclude that $u \not \equiv 0$ on $\partial M$. Now, integrating \eqref{eq:thm3-1} over $M$ yields
$$
\alpha C_{\alpha}^{\frac{4-2p}{p-1}} \int_{\partial M}|u_{\alpha}|^{p-2}\,\ud s_g=\xi_{\alpha} < S(n,\sigma)^{-1}.
$$
Note that $\frac{4-2p}{p-1}=-\frac{8\sigma}{n+2\sigma}<0$, it follows from $\lim_{\alpha\to \infty} C_{\alpha}=0$ that $\alpha C_{\alpha}^{\frac{4-2p}{p-1}} \rightarrow \infty$ as $\alpha \rightarrow \infty$. Hence, $\int_{\partial M}|u_{\alpha}|^{p-2}\,\ud s_g \rightarrow 0$ as $\alpha \rightarrow \infty$. Since $u_{\alpha} \to u$ strongly in $L^{p-2}(\partial M)$, we get that
$$
\int_{\partial M}|u|^{p-2}\,\ud s_g=0,
$$
which is impossible since $u \not \equiv 0$ on $\partial M$. This ends the proof of Proposition \ref{pro:3}.
\end{proof}

\section{Proof of Theorem \ref{thm}}

In this section, we will prove Theorem \ref{thm}. To this end, let us recall some known facts about the extremal functions of the sharp weighted Sobolev trace inequality \eqref{SharpWSTineRn}. By definition, an extremal function $U_0\not\equiv 0$ is such that it realizes the case of equality in \eqref{SharpWSTineRn}.

Given any $\varepsilon>0$ and $x_0 \in \mathbb{R}^{n}=\partial \mathbb{R}_+^{n+1}$, define
\be\label{eq:wvarsigma}
w_{\varepsilon, x_0}(x):=\alpha_{n,\sigma}\Big(\frac{\varepsilon}{\varepsilon^{2}+|x-x_0|^{2}}\Big)^{\frac{n-2 \sigma}{2}},\quad x\in \mathbb{R}^n,
\ee
where
$$
\alpha_{n, \sigma}=2^{\frac{n-2 \sigma}{2}}\Big(\frac{\Gamma((n+2 \sigma)/2)}{\Gamma((n-2 \sigma)/2)}\Big)^{\frac{n-2 \sigma}{4 \sigma}}.
$$
For all $(x,t) \in \mathbb{R}_{+}^{n+1}$, define
\be\label{eq:Wvarsigma}
W_{\varepsilon, x_0}(x, t):=p_{n, \sigma} \int_{\mathbb{R}^{n}} \frac{t^{2 \sigma}}{(|x-y|^{2}+t^{2})^{\frac{n+2 \sigma}{2}}} w_{\varepsilon, x_0}(y) \,\ud y,
\ee
where
$$
p_{n, \sigma}=\frac{\Gamma((n+2 \sigma)/{2})}{\pi^{n/2} \Gamma(\sigma)}.
$$
We say that $W_{\varepsilon, x_0}$ is the $\sigma$-harmonic extension of $w_{\varepsilon, x_0}$. The equality of \eqref{SharpWSTineRn} is attained by $U=c W_{\varepsilon, x_0}$ for any $c \in \mathbb{R}$, $\varepsilon>0$ and $x_0 \in \mathbb{R}^{n}$. If $\sigma=1/2$, the function $W_{\varepsilon, x_0}$ can be explicitly written as
$$
W_{\varepsilon, x_0}(x, t)=\alpha_{n, \frac{1}{2}}\Big(\frac{\varepsilon}{(\varepsilon+t)^{2}+|x-x_0|^{2}}\Big)^{\frac{n-1}{2}}.
$$
Moreover, for any $\sigma\in (0,1)$, we have
\be\label{eq:wvarp=}
\Big(\int_{\mathbb{R}^{n}} w_{\varepsilon, x_0}^p \,\ud x\Big)^{\frac{2 \sigma}{n}}=S(n, \sigma)^{-1} \kappa_{\sigma},
\ee
where
$$
\kappa_{\sigma}=\frac{\Gamma(\sigma)}{2^{1-2\sigma}\Gamma(1-\sigma)}.
$$

For future use, we denote $w_{\varepsilon}=w_{\varepsilon, 0}$ and $W_{\varepsilon}=W_{\varepsilon, 0}$. It can be easily checked that $W_{\varepsilon}(\cdot, t)$ is radially symmetric for each $t>0$. Moreover,
$$
w_{\varepsilon}(x)=\varepsilon^{-(n-2 \sigma) / 2} w_{1}(\varepsilon^{-1}x) \quad \text { and } \quad W_{\varepsilon}(x, t)=\varepsilon^{-(n-2 \sigma) / 2} W_{1}(\varepsilon^{-1}x, \varepsilon^{-1}t).
$$

Given $P \in \partial M$, let $x=(x^{1}, \cdots, x^{n})$ be normal coordinates on $\partial M$ at $P$ and $t=\rho$. In other words, let $(x, t)$ be Fermi coordinates at $P$. Let us recall the expansion of the metric $g$ on $M$ near the boundary $\partial M$. Its proof can be found in Escobar \cite[Lemma 3.1 and Lemma 3.2]{EscobarConformal1992}.

\begin{lemma}\label{lem:Taylor}
Suppose that $P\in \partial M$. Then for $(x,t)$ small, it holds that
\be\label{eq:detg}
\sqrt{|g|}=1-Ht+\frac{1}{2}(H^{2}-\|\pi\|^{2}-R_{tt}) t^{2}-H_{, i} x^{i}t-\frac{1}{6} \bar{R}_{i j} x^{i} x^{j}+O(|(x,t)|^{3}),
\ee
and
\be\label{eq:gij}
g^{i j} =\delta^{i j}+2 \pi^{i j} t-\frac{1}{3} \bar{R}^{i}{ }_{k l}{}^{j} x^{k} x^{l}+g^{i j}{}_{, t m}  x^{m}t +(3 \pi^{i m} \pi_{m}{ }^{j}+R^{i}{ }_{t}{}^j{}_{t}) t^{2}+O(|(x,t)|^{3}),
\ee
where $H$ is the mean curvature of $\partial M$, $\pi^{i j}$ are the components of the second fundamental form
$\pi$ of $\partial M$ and $\|\pi\|^{2}=g^{i k} g^{j l} \pi_{i j} \pi_{k l}$, $\bar{R}^{i}{ }_{k l}{}^{j}$ are the components of the induced Riemannian curvature tensor on
$\partial M$, $R^{i}{ }_{t}{}^j{}_{t}$ is that of the Riemannian curvature tensor in $M$, $\bar{R}_{i j}$ are the components of the
induced Ricci curvature tensor on $\partial M$, $R_{tt}$ is that of the Ricci curvature tensor in $M$, and $H_{, i}$, $g^{i j}{}_{, t m}$
denote the derivative of the mean curvature and metric tensor, respectively. Every tensor in the expansions is computed at $P$.
\end{lemma}

Let $\delta>0$ be a fixed small number, define $B_{\delta}$ and $\mathcal{B}^+_{\delta}$ be the $n$-dimensional ball and the $(n+1)$-dimensional upper half-ball centered at 0 whose radius is $\delta$, respectively. Let $\eta \in C_{c}^{\infty}(\mathbb{R}_{+}^{n+1})$ be a smooth radial cut-off function such that $0\leq \eta\leq 1$, $\eta\equiv1$ in $\mathcal{B}^+_{\delta}$ and $\eta\equiv 0$ in $\mathbb{R}_{+}^{n+1} \backslash \mathcal{B}^+_{2\delta}$. In the following, we will abuse notations by denoting $B_{\delta}$ and $\mathcal{B}_{\delta}$ as the geodesic ball  centered at $P$ with radius $\delta$.

Now, we define a family of test functions
\be\label{eq:defphivar}
\phi_{\varepsilon}=\eta W_{\varepsilon}-\mu_{\varepsilon}
=
\begin{cases}
W_{\varepsilon}-\mu_{\varepsilon} & \text{ in }\,  M\cap \mathcal{B}^+_{\delta},\\
\eta W_{\varepsilon}-\mu_{\varepsilon} & \text{ in }\,  M\cap (\mathcal{B}^+_{2\delta}\backslash \mathcal{B}^+_{\delta}) ,\\
-\mu_{\varepsilon} & \text{ in }\,  M\backslash \mathcal{B}^+_{2\delta},
\end{cases}
\ee
where $\mu_{\varepsilon}>0$ is chosen such that
\be\label{eq:varep=0}
\int_{\partial M}|\phi_{\varepsilon}|^{p-2} \phi_{\varepsilon}\,\ud s_g=0,
\ee
therefore, $\phi_{\varepsilon} \in \Lambda$.

In the remainder of this section, we will prove that, under the assumptions of Theorem \ref{thm}, for $\varepsilon>0$ small enough,
$$
\frac{\int_{M}\rho^{1-2\sigma}|\nabla_g \phi_{\varepsilon}|^{2}\,\ud v_g}{(\int_{\partial M}|\phi_{\varepsilon}|^{p}\,\ud s_g)^{2/p}}<S(n,\sigma)^{-1},
$$
thus by Proposition \ref{pro:3}, Theorem \ref{thm} is proved.

We first estimate $\mu_{\varepsilon}$. Thanks to the definition of $\phi_{\varepsilon}$, we have that
\be\label{eq:varepBc}
\int_{\partial M \backslash B_{2\delta}}|\phi_{\varepsilon}|^{p-2} \phi_{\varepsilon}\,\ud s_g\sim-\mu_{\varepsilon}^{\frac{n+2\sigma}{n-2\sigma}}.
\ee
On the other hand, we get from \eqref{eq:detg} that, when $\varepsilon \rightarrow 0$,
\be\label{eq:varepB}
\begin{aligned}
&\int_{\partial M \cap B_{2\delta}}|\phi_{\varepsilon}|^{p-2} \phi_{\varepsilon}\,\ud s_g\\
=& \int_{0}^{2\delta} \int_{S(r)}\Big|\eta(r)\alpha_{n,\sigma}\Big(\frac{\varepsilon}{\varepsilon^{2}+r^{2}}\Big)^{\frac{n-2\sigma}{2}}-\mu_{\varepsilon}\Big|^{\frac{4\sigma}{n-2\sigma}}
\Big(\eta(r)\alpha_{n,\sigma}\Big(\frac{\varepsilon}{\varepsilon^{2}+r^{2}}\Big)^{\frac{n-2\sigma}{2}}-\mu_{\varepsilon}\Big) \sqrt{|g|} \,\ud \sigma \,\ud r\\
\sim& \int_{0}^{2\delta} \int_{S(r)}\Big|\eta(r)\alpha_{n,\sigma}\Big(\frac{\varepsilon}{\varepsilon^{2}+r^{2}}\Big)^{\frac{n-2\sigma}{2}}-\mu_{\varepsilon}\Big|^{\frac{4\sigma}{n-2\sigma}}
\Big(\eta(r)\alpha_{n,\sigma}\Big(\frac{\varepsilon}{\varepsilon^{2}+r^{2}}\Big)^{\frac{n-2\sigma}{2}}-\mu_{\varepsilon}\Big) \,\ud \sigma \,\ud r\\
\sim&  \int_{0}^{2\delta} \Big|\eta(r)\alpha_{n,\sigma}\Big(\frac{\varepsilon}{\varepsilon^{2}+r^{2}}\Big)^{\frac{n-2\sigma}{2}}-\mu_{\varepsilon}\Big|^{\frac{4\sigma}{n-2\sigma}}
\Big(\eta(r)\alpha_{n,\sigma}\Big(\frac{\varepsilon}{\varepsilon^{2}+r^{2}}\Big)^{\frac{n-2\sigma}{2}}-\mu_{\varepsilon}\Big)r^{n-1}  \,\ud r\\
\sim& \varepsilon^{-\sigma+\frac{n}{2}}
\int_{0}^{{2\delta}/{\varepsilon}}\Big|\eta(\varepsilon s)\alpha_{n,\sigma}\Big(\frac{1}{1+s^{2}}\Big)^{\frac{n-2\sigma}{2}}-\varepsilon^{\frac{n-2\sigma}{2}} \mu_{\varepsilon}\Big|^{\frac{4\sigma}{n-2\sigma}}\Big(\eta(\varepsilon s)\alpha_{n,\sigma}\Big(\frac{1}{1+s^{2}}\Big)^{\frac{n-2\sigma}{2}}-\varepsilon^{\frac{n-2\sigma}{2}} \mu_{\varepsilon}\Big) s^{n-1} \,\ud s,
\end{aligned}
\ee
where $\ud \sigma$ is the volume element of $S(r)=\{Q \in \partial M \mid \operatorname{dist}_{g}(Q, P)=r\}$. Choosing a
subsequence if necessary, we can suppose that $\varepsilon^{\frac{n-2\sigma}{2}} \mu_{\varepsilon}\to 0$ as $\varepsilon\to 0$.

Combining \eqref{eq:varep=0}, \eqref{eq:varepBc}, and \eqref{eq:varepB}, we obtain
\be\label{eq:muvarsim}
\mu_{\varepsilon} \sim \varepsilon^{\frac{(n-2\sigma)^{2}}{2 (n+2\sigma)}}
\ee
as $\varepsilon\to 0$.

In order to prove Theorem \ref{thm}, we distinguish two cases.

{\bf Case 1.} $n=3$, $\sigma\in (0,1/2)$, $H(P) > 0$ or $n \geq 4$, $\sigma\in (0,1)$, $H(P)>0$.

{\bf Step 1:} Computation of $(\int_{\partial M} |\phi_{\varepsilon}|^p\,\ud s_g)^{-2/p}$.

By the definition of $\phi_{\varepsilon}$, we have that
$$
\begin{aligned}
&\int_{\partial M} |\phi_{\varepsilon}|^{p}\,\ud s_g\\
=&\int_{\partial M \backslash B_{2\delta}} |\phi_{\varepsilon}|^{p}\,\ud s_g+\int_{\partial M \cap B_{2\delta}} |\phi_{\varepsilon}|^{p}\,\ud s_g\\
=&\mu_{\varepsilon}^{p} \operatorname{Vol}(\partial M \backslash B_{2\delta})+\int_{0}^{2\delta} \int_{S(r)}\Big|\eta(r)\alpha_{n,\sigma}\Big(\frac{\varepsilon}{\varepsilon^{2}+r^{2}}\Big)^{\frac{n-2\sigma}{2}}-\mu_{\varepsilon}\Big|^{\frac{2n}{n-2\sigma}} \sqrt{|g|} \,\ud \sigma \,\ud r.
\end{aligned}
$$
It follows from \eqref{eq:detg} that
$$
\sqrt{|g|}(x,0)=1-\frac{\bar{R}_{i j}(P)}{6}  x^{i} x^{j}+O(|x|^{3}).
$$
Therefore,
$$
\int_{S(r)} \sqrt{|g|} \,\ud \sigma=\omega_{n-1} r^{n-1}\Big(1-\frac{\bar{R}(P)}{6n} r^{2}+o(r^{2})\Big),
$$
where $\bar{R}$ is the scalar curvature of $\partial M$. Hence,
\be\label{eq:parMexpansion}
\begin{aligned}
&\int_{\partial M} |\phi_{\varepsilon}|^{p}\,\ud s_g\\
=&\mu_{\varepsilon}^{p} \operatorname{Vol}(\partial M \backslash B_{2\delta})+\int_{0}^{2\delta} \int_{S(r)}\Big|\eta(r)\alpha_{n,\sigma}\Big(\frac{\varepsilon}{\varepsilon^{2}+r^{2}}\Big)^{\frac{n-2\sigma}{2}}-\mu_{\varepsilon}\Big|^{\frac{2n}{n-2\sigma}} \sqrt{|g|} \,\ud \sigma \,\ud r\\
=&\mu_{\varepsilon}^{p} \operatorname{Vol}(\partial M \backslash B_{2\delta}) \\
&+\omega_{n-1} \int_{0}^{{2\delta}}\Big|\eta(r)\alpha_{n,\sigma}\Big(\frac{\varepsilon}{\varepsilon^{2}+r^{2}}\Big)^{\frac{n-2\sigma}{2}}-\mu_{\varepsilon}\Big|
^{\frac{2n}{n-2\sigma}} r^{n-1}
\Big(1-\frac{\bar{R}(P)}{6n}r^{2}+o(r^{2})\Big) \,\ud r \\
=&\mu_{\varepsilon}^{p} \operatorname{Vol}(\partial M \backslash B_{2\delta}) \\
&+\omega_{n-1} \int_{0}^{{2\delta}/{\varepsilon}}\Big|\eta(\varepsilon s)\alpha_{n,\sigma}\Big(\frac{1}{1+s^{2}}\Big)^{\frac{n-2\sigma}{2}}-\varepsilon^{\frac{n-2\sigma}{2}} \mu_{\varepsilon}\Big|^{\frac{2n}{n-2\sigma}} s^{n-1}
\Big(1-\frac{\bar{R}(P)}{6n}\varepsilon^{2} s^{2}+\varepsilon^{2} o(s^{2})\Big) \,\ud s \\
=&\mu_{\varepsilon}^{p} \operatorname{Vol}(\partial M \backslash B_{2\delta}) +\omega_{n-1} \alpha_{n,\sigma}^{{2n}/(n-2\sigma)} \int_{0}^{\infty}\frac{s^{n-1}}{(1+s^{2})^{n}}\Big(1-\frac{\bar{R}(P)}{6n}\varepsilon^{2} s^{2}+\varepsilon^{2} o(s^{2})\Big) \,\ud s+{O(\varepsilon^{n})}.
\end{aligned}
\ee
By the change of variables $s=\varepsilon^{-1} r$, \eqref{eq:wvarsigma} and \eqref{eq:wvarp=}, we have
\be\label{eq:intparts1}
\begin{aligned}
\omega_{n-1}\alpha_{n,\sigma}^{{2n}/(n-2\sigma)} \int_{0}^{\infty} \frac{s^{n-1}}{(1+s^{2})^{n}}\,\ud s=&\omega_{n-1}\alpha_{n,\sigma}^{{2n}/(n-2\sigma)}\int_{0}^{\infty} \frac{\varepsilon^n}{(\varepsilon^2+r^{2})^{n}}r^{n-1}\,\ud r \\
=&\alpha_{n,\sigma}^{{2n}/(n-2\sigma)}\int_{\mathbb{R}^{n}} \Big(\frac{\varepsilon}{\varepsilon^2+|x|^2}\Big)^n\,\ud x\\
=&\int_{\mathbb{R}^{n}} w_{\varepsilon}^{p}\,\ud x\\
=&({S(n,\sigma)^{-1}}\kappa_{\sigma})^{n/2\sigma}.
\end{aligned}
\ee
Using integration by parts, we have
$$
\begin{aligned}
\int_{0}^{\infty} \frac{s^{n+1}}{(1+s^{2})^{n}}\, \ud s=&\int_{0}^{\infty} \frac{(1+s^2)s^{n-1}-s^{n-1}}{(1+s^{2})^{n}}\, \ud s\\
=&\int_{0}^{\infty} \frac{s^{n-1}}{(1+s^{2})^{n-1}}\, \ud s-\int_{0}^{\infty} \frac{s^{n-1}}{(1+s^{2})^{n}}\, \ud s\\
=&\frac{2(n-1)}{n}\int_{0}^{\infty} \frac{s^{n+1}}{(1+s^{2})^{n}}\, \ud s-\int_{0}^{\infty} \frac{s^{n-1}}{(1+s^{2})^{n}}\, \ud s,
\end{aligned}
$$
therefore,
\be\label{eq:intparts2}
\begin{aligned}
\int_{0}^{\infty} \frac{s^{n+1}}{(1+s^{2})^{n}}\, \ud s=& \frac{n}{n-2} \int_{0}^{\infty} \frac{s^{n-1}}{(1+s^{2})^{n}} \,\ud s\\
=&\frac{n}{(n-2)\omega_{n-1}\alpha_{n,\sigma}^{{2n}/(n-2\sigma)}}({S(n,\sigma)^{-1}}\kappa_{\sigma})^{n/2\sigma}.
\end{aligned}
\ee

Now inserting \eqref{eq:muvarsim}, \eqref{eq:intparts1}, \eqref{eq:intparts2} into \eqref{eq:parMexpansion}, we obtain
$$
\begin{aligned}
&\int_{\partial M} |\phi_{\varepsilon}|^{p}\,\ud s_g\\
=&\mu_{\varepsilon}^{p} \operatorname{Vol}(\partial M \backslash B_{2\delta}) +\omega_{n-1} \alpha_{n,\sigma}^{{2n}/(n-2\sigma)} \int_{0}^{\infty}\frac{s^{n-1}}{(1+s^{2})^{n}}\Big(1-\frac{\bar{R}(P)}{6n}\varepsilon^{2} s^{2}+\varepsilon^{2} o(s^{2})\Big) \,\ud s+O(\varepsilon^{n})\\
=&O(\varepsilon^{\frac{n(n-2\sigma)}{n+2\sigma}}) +({S(n,\sigma)^{-1}}\kappa_{\sigma})^{n/2\sigma}-\frac{\bar{R}(P)}{6(n-2)}({S(n,\sigma)^{-1}}\kappa_{\sigma})^{n/2\sigma}\varepsilon^{2}
+o(\varepsilon^{2})+O(\varepsilon^{n})\\
=&({S(n,\sigma)^{-1}}\kappa_{\sigma})^{n/2\sigma}\Big(1-\frac{\bar{R}(P)}{6(n-2)} \varepsilon^{2}+o(\varepsilon^{2})+O(\varepsilon^n)+O(\varepsilon^{\frac{n(n-2\sigma)}{n+2\sigma}})\Big).
\end{aligned}
$$
If $n=3$, we have
$$
\int_{\partial M} |\phi_{\varepsilon}|^{p}\,\ud s_g=
\begin{cases}
({S(n,\sigma)^{-1}}\kappa_{\sigma})^{n/2\sigma}(1-\frac{\bar{R}(P)}{6(n-2)} \varepsilon^{2}+o(\varepsilon^{2})) &\text{ if }\, 0<\sigma<3/10,\\
({S(n,\sigma)^{-1}}\kappa_{\sigma})^{n/2\sigma}(1+O(\varepsilon^{\frac{3(3-2\sigma)}{3+2\sigma}})) &\text{ if }\, 3/10\leq \sigma<1.
\end{cases}
$$
Hence,
\be\label{eq:step1n=3}
\begin{aligned}
&\Big(\int_{\partial M} |\phi_{\varepsilon}|^{p}\,\ud s_g\Big)^{-2/p}\\
=&
\begin{cases}
({S(n,\sigma)^{-1}}\kappa_{\sigma})^{-(n-2\sigma)/2\sigma}(1+\frac{(n-2\sigma)\bar{R}(P)}{6n(n-2)}  \varepsilon^{2}+o(\varepsilon^{2})) &\text{ if }\, 0<\sigma<3/10,\\
({S(n,\sigma)^{-1}}\kappa_{\sigma})^{-(n-2\sigma)/2\sigma}(1+O(\varepsilon^{\frac{3(3-2\sigma)}{3+2\sigma}})) &\text{ if }\, 3/10\leq \sigma<1.
\end{cases}
\end{aligned}
\ee
Similarly, we have, if $n=4$,
\be\label{eq:step1n=4}
\begin{aligned}
&\Big(\int_{\partial M} |\phi_{\varepsilon}|^{p}\,\ud s_g\Big)^{-2/p}\\
=&
\begin{cases}
({S(n,\sigma)^{-1}}\kappa_{\sigma})^{-(n-2\sigma)/2\sigma}(1+\frac{(n-2\sigma)\bar{R}(P)}{6n(n-2)}  \varepsilon^{2}+o(\varepsilon^{2})) &\text{ if }\, 0<\sigma<2/3,\\
({S(n,\sigma)^{-1}}\kappa_{\sigma})^{-(n-2\sigma)/2\sigma}(1+O(\varepsilon^{\frac{4(2-\sigma)}{2+\sigma}})) &\text{ if }\, 2/3\leq \sigma<1.
\end{cases}
\end{aligned}
\ee
If $n\geq5$,
\be\label{eq:step1n>=5}
\Big(\int_{\partial M} |\phi_{\varepsilon}|^{p}\,\ud s_g\Big)^{-2/p}=({S(n,\sigma)^{-1}}\kappa_{\sigma})^{-\frac{n-2\sigma}{2\sigma}}\Big(1+\frac{(n-2\sigma)\bar{R}(P)}{6n(n-2)} \varepsilon^{2}+o(\varepsilon^{2})\Big).
\ee

By the definition \eqref{eq:defphivar} of the function $\phi_{\varepsilon}$,
\be\label{eq:I1+I2}
\begin{aligned}
\int_{M}\rho^{1-2\sigma}|\nabla_g \phi_{\varepsilon}|^{2}\,\ud v_g&=\int_{\mathcal{B}_{2\delta}^{+}}\rho^{1-2\sigma}|\nabla_g \phi_{\varepsilon}|^{2}\,\ud v_g\\
&= \int_{\mathcal{B}_{\delta}^{+}}t^{1-2\sigma}|\nabla_g W_{\varepsilon}|^{2}\,\ud v_g+\int_{\mathcal{B}_{2\delta}^{+}\backslash \mathcal{B}_{\delta}^{+}}t^{1-2\sigma}|\nabla_g \phi_{\varepsilon}|^{2}\,\ud v_g\\
&=:I_1+I_2.
\end{aligned}
\ee
For the rest of this paper, we set
$$
|\nabla U|^2=(\partial_{x_1}U)^2+\cdots+(\partial_{x_n}U)^2+(\partial_t U)^2, \quad |\nabla_x U|^2=(\partial_{x_1}U)^2+\cdots+(\partial_{x_n}U)^2.
$$

{\bf Step 2:} Computation of $I_1$.

Before starting the computation, let us make one useful observation. For any $k \in \mathbb{N}$, we get from Lemma \ref{lem:appendix1-1} that
\be\label{eq:important}
\int_{\mathcal{B}_{\delta}^{+}} t^{1-2 \sigma}(|(x,t)|^{k}|\nabla W_{\varepsilon}|^{2}) \,\ud x\,\ud t\\
=
\begin{cases}
O(\varepsilon^{k}) & \text { if }\, n>2 \sigma+k, \\
O(\varepsilon^{k} \log (\delta / \varepsilon)) & \text { if }\, n=2 \sigma+k, \\
O(\varepsilon^{k}(\delta / \varepsilon)^{2 \sigma+k-n}) & \text { if }\, n<2 \sigma+k.
\end{cases}
\ee

Using \eqref{eq:gij}, we get
$$
\begin{aligned}
I_1=&\int_{\mathcal{B}_{\delta}^{+}}t^{1-2\sigma}|\nabla_g W_{\varepsilon}|^{2}\,\ud v_g\\
=&\int_{\mathcal{B}_{\delta}^{+}} t^{1-2\sigma}({g}^{i j}  \partial_i W_{\varepsilon}(x, t)  \partial_j W_{\varepsilon}(x, t)+(\partial_{t} W_{\varepsilon}(x, t))^{2}) \,\ud v_g\\
=&\int_{\mathcal{B}_{\delta}^{+}} t^{1-2\sigma}|\nabla W_{\varepsilon}|^2 \,\ud v_g+2 \pi^{i j}(P) \int_{\mathcal{B}_{\delta}^{+}} t^{2-2\sigma} \partial_i W_{\varepsilon} \partial_j W_{\varepsilon} \,\ud v_{g}\\
&+\int_{\mathcal{B}_{\delta}^{+}}t^{1-2\sigma} O(|(x,t)|^{2})|\nabla_x W_{\varepsilon}|^{2}\,\ud  v_{g}.
\end{aligned}
$$
By \eqref{eq:detg}, we have
\be\label{eq:H>0-1}
\begin{aligned}
I_1= &\int_{\mathcal{B}_{\delta}^{+}} t^{1-2\sigma}|\nabla W_{\varepsilon}|^2\,\ud x \,\ud t-H(P) \int_{\mathcal{B}_{\delta}^{+}} t^{2-2\sigma}|\nabla W_{\varepsilon}|^2\,\ud x \,\ud t\\
&+2 \pi^{i j}(P) \int_{\mathcal{B}_{\delta}^{+}} t^{2-2\sigma} \partial_i W_{\varepsilon} \partial_j W_{\varepsilon} \,\ud x\,\ud t+\int_{\mathcal{B}_{\delta}^{+}} t^{1-2\sigma} O(|(x,t)|^2) |\nabla W_{\varepsilon}|^2 \,\ud x\,\ud t.
\end{aligned}
\ee
Since $W_{\varepsilon}$ is the extremal function of \eqref{SharpWSTineRn}, it follows from \eqref{eq:wvarp=} that
\be\label{eq:H>0-2}
\begin{aligned}
\int_{\mathcal{B}_{\delta}^{+}} t^{1-2\sigma}|\nabla W_{\varepsilon}|^2\,\ud x \,\ud t\leq&\int_{\mathbb{R}_{+}^{n+1}} t^{1-2 \sigma}|\nabla W_{\varepsilon}|^{2} \,\ud x \,\ud t\\
=&S(n,\sigma)^{-1}\Big(\int_{\mathbb{R}^{n}} w_{\varepsilon}^{\frac{2 n}{n-2 \sigma}} \,\ud x\Big)^{\frac{n-2\sigma}{n}}\\
=&\kappa_{\sigma}^{\frac{n-2\sigma}{2\sigma}}S(n,\sigma)^{-\frac{n}{2\sigma}}.
\end{aligned}
\ee
On the other hand, since $\partial_i W_{\varepsilon}$ is odd in $x_{i}$ and $\pi^{i j} \delta_{i j}=H$ at the point $P$, using \eqref{eq:important} and the positivity of $H(P)$, it holds that
\be\label{eq:H>0-3}
\begin{aligned}
&-H(P) \int_{\mathcal{B}_{\delta}^{+}} t^{2-2\sigma}|\nabla W_{\varepsilon}|^2\,\ud x \,\ud t+2 \pi^{i j}(P) \int_{\mathcal{B}_{\delta}^{+}} t^{2-2\sigma} \partial_i W_{\varepsilon} \partial_j W_{\varepsilon} \,\ud x\,\ud t\\
=&-H(P) \int_{\mathcal{B}_{\delta}^{+}} t^{2-2\sigma}|\nabla W_{\varepsilon}|^2\,\ud x \,\ud t+\frac{2}{n}H(P)\int_{\mathcal{B}_{\delta}^{+}} t^{2-2\sigma}|\nabla_x W_{\varepsilon}|^2 \,\ud x\,\ud t\\
\leq&-H(P) \int_{\mathcal{B}_{\delta}^{+}} t^{2-2\sigma}|\nabla W_{\varepsilon}|^2\,\ud x \,\ud t+\frac{2}{n}H(P)\int_{\mathcal{B}_{\delta}^{+}} t^{2-2\sigma}|\nabla W_{\varepsilon}|^2 \,\ud x\,\ud t\\
\sim&\frac{2-n}{n}H(P)\varepsilon\\
\leq&-CH(P)\varepsilon,
\end{aligned}
\ee
where $C>0$. Applying \eqref{eq:important} again, we get
\be\label{eq:H>0-4}
\int_{\mathcal{B}_{\delta}^{+}} t^{1-2\sigma} O(|(x,t)|^2) |\nabla W_{\varepsilon}|^2 \,\ud x\,\ud t=O(\varepsilon^{2}),
\ee
if $n>2\sigma+2$.

Inserting \eqref{eq:H>0-2}, \eqref{eq:H>0-3} and \eqref{eq:H>0-4} into \eqref{eq:H>0-1}, we obtain
\be\label{eq:case1I1}
I_{1} \leq \kappa_{\sigma}^{\frac{n-2\sigma}{2\sigma}}S(n,\sigma)^{-\frac{n}{2\sigma}}-CH(P)\varepsilon+O(\varepsilon^{2}),
\ee
where $C>0$.

{\bf Step 3:} Computation of $I_2$.

Note that
$$
|\nabla_{g} \phi_{\varepsilon}|^{2} \leq C|\nabla \phi_{\varepsilon}|^{2} \leq C(\eta^{2}|\nabla W_{\varepsilon}|^{2}+W_{\varepsilon}^{2}|\nabla \eta|^{2}),
$$
so that, because of the structure of the cut-off function $\eta$,
\be\label{eq:case1step3-1}
|\nabla_{g} \phi_{\varepsilon}|^{2} \leq C|\nabla W_{\varepsilon}|^{2}+\frac{C}{\delta^2}W_{\varepsilon}^{2}.
\ee
Moreover, using the fact that $W_{\varepsilon}(\varepsilon x, \varepsilon t)=\varepsilon^{-(n-2 \sigma) / 2} W_{1}(x, t)$, we have, if $n>2\sigma+2$,
\be\label{eq:case1step3-2}
\int_{\mathcal{B}_{2 \delta}^{+} \backslash \mathcal{B}_{\delta}^{+}} t^{1-2\sigma}W_{\varepsilon}^{2} \,\ud x \,\ud t = \varepsilon^{2} \int_{\mathcal{B}_{2 \delta / \varepsilon}^{+} \backslash \mathcal{B}_{\delta / \varepsilon}^{+}} t^{1-2\sigma}W_{1}^{2}\,\ud x \,\ud t=\varepsilon^{2} o(1),
\ee
because by Lemma \ref{lem:appendix1-1}, the integral $\int_{\mathbb{R}_+^{n+1}} t^{1-2\sigma}W_{1}^{2} \,\ud x \,\ud t$ is finite and $\delta / \varepsilon \rightarrow \infty$ as $\varepsilon\to 0$.

On the other hand, if $n>2\sigma+2$, then using Lemma \ref{lem:appendix1-1} we have
$$
\Big(\frac{\delta}{\varepsilon}\Big)^{2} \int_{\mathcal{B}_{2 \delta / \varepsilon}^{+} \backslash \mathcal{B}_{\delta / \varepsilon}^{+}} t^{1-2\sigma}|\nabla W_{1}|^{2} \,\ud x \,\ud t \leq \int_{\mathcal{B}_{2 \delta / \varepsilon}^{+} \backslash \mathcal{B}_{\delta / \varepsilon}^{+}} t^{1-2\sigma}|(x, t)|^2|\nabla W_{1}|^{2} \,\ud x \,\ud t<\infty.
$$
Hence,
\be\label{eq:case1step3-3}
\int_{\mathcal{B}_{2 \delta}^{+} \backslash \mathcal{B}_{\delta}^{+}} t^{1-2\sigma}|\nabla W_{\varepsilon}|^{2} \,\ud x \,\ud t =\int_{\mathcal{B}_{2 \delta/\varepsilon}^{+} \backslash \mathcal{B}_{\delta/\varepsilon}^{+}} t^{1-2\sigma}|\nabla W_{1}|^{2} \,\ud x \,\ud t= O(\varepsilon^2).
\ee
Putting together \eqref{eq:case1step3-1}, \eqref{eq:case1step3-2}, and \eqref{eq:case1step3-3}, we obtain
\be\label{eq:case1I2}
I_2=\int_{\mathcal{B}_{2 \delta}^{+} \backslash \mathcal{B}_{\delta}^{+}} t^{1-2\sigma}|\nabla_g \phi_{\varepsilon}|^{2} \,\ud v_g = o(\varepsilon).
\ee

{\bf Step 4:} Conclusion.

Firstly, \eqref{eq:I1+I2}, \eqref{eq:case1I1}, and \eqref{eq:case1I2} imply that
$$
\int_{M}\rho^{1-2\sigma}|\nabla_g \phi_{\varepsilon}|^{2}\,\ud v_g\leq \kappa_{\sigma}^{\frac{n-2\sigma}{2\sigma}}S(n,\sigma)^{-\frac{n}{2\sigma}}-CH(P)\varepsilon+o(\varepsilon).
$$
Moreover, if $n=3$, using \eqref{eq:step1n=3} we have
$$
\frac{\int_{M}\rho^{1-2\sigma}|\nabla_g \phi_{\varepsilon}|^{2}\,\ud v_g}{(\int_{\partial M} |\phi_{\varepsilon}|^{p}\,\ud s_g)^{2/p}}=
S(n,\sigma)^{-1}-CH(P)\varepsilon+o(\varepsilon) \quad \text{ if }\, 0<\sigma<1/2.
$$
If $n\geq 4$, it follows from \eqref{eq:step1n=4} and \eqref{eq:step1n>=5} that
$$
\frac{\int_{M}\rho^{1-2\sigma}|\nabla_g \phi_{\varepsilon}|^{2}\,\ud v_g}{(\int_{\partial M} |\phi_{\varepsilon}|^{p}\,\ud s_g)^{2/p}}=
S(n,\sigma)^{-1}-CH(P)\varepsilon+o(\varepsilon) \quad \text{ if }\, 0<\sigma<1.
$$
Using the assumption of the mean curvature $H(P)$, we obtain for $\varepsilon$ small enough,
$$
\frac{\int_{M}\rho^{1-2\sigma}|\nabla_g \phi_{\varepsilon}|^{2}\,\ud v_g}{(\int_{\partial M} |\phi_{\varepsilon}|^{p}\,\ud s_g)^{2/p}}<S(n,\sigma)^{-1}.
$$
Using Proposition \ref{pro:3}, this ends the proof of Theorem \ref{thm} in this case.

In the rest of this section we deal with the second case.

{\bf Case 2.} $n \geq 5$, $\sigma\in (0,1)$, $H(P)=0$ and
$$
\frac{3n^2-6n-4\sigma^2+4}{12(1-\sigma)(n-1)(n-2-2\sigma)}\bar{R}(P)+\|\pi\|^2(P)+
\frac{3n-2-2\sigma}{3n-6-6\sigma}R_{tt}(P)>0.
$$

Let us still use the notation of the first case.

{\bf Step 1:} Computation of $I_1$.

Notice that we may assume that $P$ is a maximum point of $H$. Indeed, if there exists $Q$ such that $H(Q)>H(P)=0$, we can finish the proof using Case 1. Therefore, $H_{,i}(P)=0$ for $1 \leq i \leq n$. It follows from \eqref{eq:detg} that
\be\label{eq:detgH=0}
\sqrt{|g|}(x,t)=1-\frac{1}{2}(\|\pi\|^{2}+R_{tt}) t^{2}-\frac{1}{6} \bar{R}_{i j} x^{i} x^{j}+O(|(x,t)|^{3}).
\ee
Similar to Case 1, using \eqref{eq:gij} we have
$$
\begin{aligned}
I_1&=\int_{\mathcal{B}_{\delta}^{+}}t^{1-2\sigma}|\nabla_g W_{\varepsilon}|^{2}\,\ud v_g\\
&=\int_{\mathcal{B}_{\delta}^{+}} t^{1-2\sigma}({g}^{i j}  \partial_{i}W_{\varepsilon}(x, t)  \partial_{j}W_{\varepsilon}(x, t)+(\partial_{t} W_{\varepsilon}(x, t))^{2}) \,\ud v_g\\
&=:J_1+J_2+J_3+J_4+J_5+J_6,
\end{aligned}
$$
where
$$
\begin{aligned}
J_1=&\int_{\mathcal{B}_{\delta}^{+}} t^{1-2\sigma}|\nabla W_{\varepsilon}|^2 \,\ud v_g,\\
J_2=&2 \pi^{i j}(P) \int_{\mathcal{B}_{\delta}^{+}} t^{2-2\sigma} \partial_{i}W_{\varepsilon} \partial_{j}W_{\varepsilon} \,\ud v_{g},\\
J_3=&-\frac{1}{3}\bar{R}^{i}{ }_{k l}{}^{j}(P) \int_{\mathcal{B}_{\delta}^{+}} t^{1-2\sigma} x^{k} x^{l} \partial_{i}W_{\varepsilon} \partial_{j}W_{\varepsilon}\,\ud  v_{g},\\
J_4=&g^{i j}{}_{, t m}(P) \int_{\mathcal{B}_{\delta}^{+}} t^{2-2\sigma} x^{m} \partial_{i}W_{\varepsilon}\partial_{j}W_{\varepsilon} \,\ud v_{g},\\
J_5=&(3 \pi^{i m}(P) \pi_{m}{ }^{j}(P)+R^{i}{ }_{t}{}^j{}_{t}(P)) \int_{\mathcal{B}_{\delta}^{+}} t^{3-2\sigma} \partial_{i}W_{\varepsilon} \partial_{j}W_{\varepsilon} \,\ud  v_{g},\\
J_6=&\int_{\mathcal{B}_{\delta}^{+}}t^{1-2\sigma} O(|(x,t)|^{3})|\nabla_x W_{\varepsilon}|^{2}\,\ud  v_{g}.
\end{aligned}
$$

In the following, we estimate these terms separately.

Firstly, by \eqref{eq:detgH=0} and \eqref{eq:important}, we have
$$
\begin{aligned}
J_1=&\int_{\mathcal{B}_{\delta}^{+}} t^{1-2\sigma}|\nabla W_{\varepsilon}|^2 \,\ud v_g\\
=&\int_{\mathcal{B}_{\delta}^{+}} t^{1-2\sigma}|\nabla W_{\varepsilon}|^2 \,\ud x\,\ud t-\frac{1}{2}(\|\pi\|^{2}+R_{tt})(P)\int_{\mathcal{B}_{\delta}^{+}} t^{3-2\sigma}|\nabla W_{\varepsilon}|^2 \,\ud x\,\ud t\\
&-\frac{1}{6} \bar{R}_{i j}(P)\int_{\mathcal{B}_{\delta}^{+}} t^{1-2\sigma}x^{i} x^{j}|\nabla W_{\varepsilon}|^2 \,\ud x\,\ud t+\int_{\mathcal{B}_{\delta}^{+}} t^{1-2\sigma}O(|(x,t)|^{3})|\nabla W_{\varepsilon}|^2 \,\ud x\,\ud t\\
=&\int_{\mathcal{B}_{\delta}^{+}} t^{1-2\sigma}|\nabla W_{\varepsilon}|^2 \,\ud x\,\ud t-\frac{1}{2}(\|\pi\|^{2}+R_{tt})(P)\int_{\mathcal{B}_{\delta}^{+}} t^{3-2\sigma}|\nabla W_{\varepsilon}|^2 \,\ud x\,\ud t\\
&-\frac{1}{6n} \bar{R}(P)\int_{\mathcal{B}_{\delta}^{+}} t^{1-2\sigma}|x|^2|\nabla W_{\varepsilon}|^2 \,\ud x\,\ud t+O(\varepsilon^3)\\
=&\int_{\mathcal{B}_{\delta}^{+}} t^{1-2\sigma}|\nabla W_{\varepsilon}|^2 \,\ud x\,\ud t-\frac{1}{2}(\|\pi\|^{2}+R_{tt})(P)\varepsilon^2\int_{\mathcal{B}_{\delta/\varepsilon}^{+}} t^{3-2\sigma}|\nabla W_{1}|^2 \,\ud x\,\ud t\\
&-\frac{1}{6n} \bar{R}(P)\varepsilon^2\int_{\mathcal{B}_{\delta/\varepsilon}^{+}} t^{1-2\sigma}|x|^2|\nabla W_{1}|^2 \,\ud x\,\ud t+O(\varepsilon^3)
\end{aligned}
$$
if $n>2\sigma+3$.

Now we try to estimate the second term $J_2$. By \eqref{eq:detgH=0} and \eqref{eq:important}, we have
$$
\begin{aligned}
J_2=&2 \pi^{i j}(P) \int_{\mathcal{B}_{\delta}^{+}} t^{2-2\sigma} \partial_i W_{\varepsilon} \partial_j W_{\varepsilon} \,\ud v_{g}\\
\leq&2 \pi^{i j}(P) \int_{\mathcal{B}_{\delta}^{+}} t^{2-2\sigma} \partial_i W_{\varepsilon} \partial_j W_{\varepsilon} \,\ud x\,\ud t+C\int_{\mathcal{B}_{\delta}^{+}} t^{1-2\sigma}|(x,t)|^3 |\nabla W_{\varepsilon}|^2  \,\ud x\,\ud t\\
=&2 \pi^{i j}(P) \int_{\mathcal{B}_{\delta}^{+}} t^{2-2\sigma} \partial_i W_{\varepsilon} \partial_j W_{\varepsilon} \,\ud x\,\ud t+O(\varepsilon^3).
\end{aligned}
$$
Since $\partial_i W_{\varepsilon}$ is odd in $x_{i}$ and $\pi^{i j}(P) \delta_{i j}=H(P)=0$, it holds
$$
2 \pi^{i j}(P) \int_{\mathcal{B}_{\delta}^{+}} t^{2-2\sigma} \partial_i W_{\varepsilon} \partial_j W_{\varepsilon} \,\ud x\,\ud t=0.
$$
Therefore,
$$
J_2=O(\varepsilon^3).
$$

Again using \eqref{eq:detgH=0} and \eqref{eq:important}, we have that

$$
\begin{aligned}
J_3=&-\frac{1}{3}\bar{R}^{i}{ }_{k l}{}^{j}(P) \int_{\mathcal{B}_{\delta}^{+}} t^{1-2\sigma} x^{k} x^{l} \partial_i W_{\varepsilon} \partial_j W_{\varepsilon}\,\ud  v_{g}\\
\leq &-\frac{1}{3}\bar{R}^{i}{ }_{k l}{}^{j}(P) \int_{\mathcal{B}_{\delta}^{+}} t^{1-2\sigma} x^{k} x^{l} \partial_i  W_{\varepsilon} \partial_j W_{\varepsilon}\,\ud x\,\ud t+C\int_{\mathcal{B}_{\delta}^{+}} t^{1-2\sigma}|(x,t)|^4 |\nabla W_{\varepsilon}|^2 \,\ud x\,\ud t\\
\leq &-\frac{1}{3}\bar{R}^{i}{ }_{k l}{}^{j}(P) \int_{\mathcal{B}_{\delta}^{+}} t^{1-2\sigma} x^{k} x^{l} \partial_i  W_{\varepsilon} \partial_j W_{\varepsilon}\,\ud x\,\ud t+C\delta\int_{\mathcal{B}_{\delta}^{+}} t^{1-2\sigma}|(x,t)|^3 |\nabla W_{\varepsilon}|^2 \,\ud x\,\ud t\\
= &-\frac{1}{3}\bar{R}^{i}{ }_{k l}{}^{j}(P) \int_{\mathcal{B}_{\delta}^{+}} t^{1-2\sigma} x^{k} x^{l} \partial_i  W_{\varepsilon} \partial_j W_{\varepsilon}\,\ud x\,\ud t+O(\varepsilon^3).
\end{aligned}
$$
It follows from the symmetries of the curvature tensor that
$$
-\frac{1}{3}\bar{R}^{i}{ }_{k l}{}^{j}(P) \int_{\mathcal{B}_{\delta}^{+}} t^{1-2\sigma} x^{k} x^{l} \partial_i  W_{\varepsilon} \partial_j W_{\varepsilon}\,\ud x\,\ud t=0.
$$
Hence,
$$
J_3=O(\varepsilon^3).
$$

Next, the calculations for $J_4$, $J_5$, and $J_6$ are very similar to the previous one. Indeed,
$$
\begin{aligned}
J_4=&g^{i j}{}_{, t m}(P) \int_{\mathcal{B}_{\delta}^{+}} t^{2-2\sigma} x^{m} \partial_i W_{\varepsilon} \partial_j W_{\varepsilon} \,\ud v_{g}\\
\leq&g^{i j}{}_{, t m}(P) \int_{\mathcal{B}_{\delta}^{+}} t^{2-2\sigma} x^{m} \partial_i W_{\varepsilon} \partial_j W_{\varepsilon} \,\ud x\,\ud t+C\int_{\mathcal{B}_{\delta}^{+}} t^{1-2\sigma}|(x,t)|^4 |\nabla W_{\varepsilon}|^2 \,\ud x\,\ud t\\
=&O(\varepsilon^3),
\end{aligned}
$$
$$
\begin{aligned}
J_5=&(3 \pi^{i m}(P) \pi_{m}{ }^{j}(P)+R^{i}{ }_{t}{}^j{}_{t}(P)) \int_{\mathcal{B}_{\delta}^{+}} t^{3-2\sigma} \partial_i W_{\varepsilon} \partial_j W_{\varepsilon} \,\ud  v_{g}\\
\leq&(3 \pi^{i m}(P) \pi_{m}{ }^{j}(P)+R^{i}{ }_{t}{}^j{}_{t}(P)) \int_{\mathcal{B}_{\delta}^{+}} t^{3-2\sigma} \partial_i W_{\varepsilon} \partial_j W_{\varepsilon} \,\ud x\,\ud t\\
&+C\int_{\mathcal{B}_{\delta}^{+}} t^{1-2\sigma}|(x,t)|^4 |\nabla W_{\varepsilon}|^2 \,\ud x\,\ud t\\
=&\frac{(3\|\pi\|^{2}+R_{tt})(P)}{n} \int_{\mathcal{B}_{\delta}^{+}} t^{3-2\sigma} |\nabla_x W_{\varepsilon}|^2 \,\ud x\,\ud t+ O(\varepsilon^3)\\
=&\frac{(3\|\pi\|^{2}+R_{tt})(P)}{n}\varepsilon^2 \int_{\mathcal{B}_{\delta/\varepsilon}^{+}} t^{3-2\sigma} |\nabla_x W_{1}|^2 \,\ud x\,\ud t+ O(\varepsilon^3),
\end{aligned}
$$
and
$$
J_6=\int_{\mathcal{B}_{\delta}^{+}}t^{1-2\sigma} O(|(x,t)|^{3})|\nabla_x W_{\varepsilon}|^{2}\,\ud  v_{g}= O(\varepsilon^3).
$$

By the above estimates and \eqref{eq:H>0-2}, we obtain
$$
\begin{aligned}
I_1=&J_1+J_2+J_3+J_4+J_5+J_6\\
\leq&\kappa_{\sigma}^{\frac{n-2\sigma}{2\sigma}}S(n,\sigma)^{-\frac{n}{2\sigma}}-\frac{1}{2}(\|\pi\|^{2}+R_{tt})(P)\varepsilon^2\int_{\mathcal{B}_{\delta/\varepsilon}^{+}} t^{3-2\sigma}|\nabla W_{1}|^2 \,\ud x\,\ud t\\
&-\frac{1}{6n} \bar{R}(P)\varepsilon^2\int_{\mathcal{B}_{\delta/\varepsilon}^{+}} t^{1-2\sigma}|x|^2|\nabla W_{1}|^2 \,\ud x\,\ud t\\
&+\frac{(3\|\pi\|^{2}+R_{tt})(P)}{n}\varepsilon^2 \int_{\mathcal{B}_{\delta/\varepsilon}^{+}} t^{3-2\sigma} |\nabla_x W_{1}|^2 \,\ud x\,\ud t+ O(\varepsilon^3).
\end{aligned}
$$
It follows from Lemma \ref{lem:appendix1-2} that
\be\label{eq:case2I1}
\begin{aligned}
I_1\leq&\kappa_{\sigma}^{\frac{n-2\sigma}{2\sigma}}S(n,\sigma)^{-\frac{n}{2\sigma}}-\frac{1}{2}(\|\pi\|^{2}+R_{tt})(P)\varepsilon^2(2(1-\sigma) A_{0}+o(1))\\
&-\frac{1}{6n} \bar{R}(P)\varepsilon^2\Big(\frac{n(n^{2}-4 n(1-\sigma)+4(1-\sigma-\sigma^{2}))}{4 \sigma(n-1)} A_{0}+o(1)\Big)\\
&+\frac{(3\|\pi\|^{2}+R_{tt})(P)}{n}\varepsilon^2 \Big(\frac{2(1-\sigma^{2})}{3} A_{0}+o(1)\Big)+ O(\varepsilon^3)\\
=&\kappa_{\sigma}^{\frac{n-2\sigma}{2\sigma}}S(n,\sigma)^{-\frac{n}{2\sigma}}-\frac{n^{2}-4 n(1-\sigma)+4(1-\sigma-\sigma^{2})}{24 \sigma(n-1)}\bar{R}(P) A_{0}\varepsilon^2\\
&-\frac{(1-\sigma)(n-2-2\sigma)}{n}\|\pi\|^{2}(P)A_{0}\varepsilon^2\\
&-\frac{(1-\sigma)(3n-2-2\sigma)}{3n}R_{tt}(P)A_{0}\varepsilon^2 + o(\varepsilon^2).
\end{aligned}
\ee

{\bf Step 2:} Computation of $I_2$.

By Lemma \ref{lem:appendix1-1}, we have
$$
\Big(\frac{\delta}{\varepsilon}\Big)^{3} \int_{\mathcal{B}_{2 \delta / \varepsilon}^{+} \backslash \mathcal{B}_{\delta / \varepsilon}^{+}} t^{1-2\sigma}|\nabla W_{1}|^{2} \,\ud x \,\ud t \leq \int_{\mathcal{B}_{2 \delta / \varepsilon}^{+} \backslash \mathcal{B}_{\delta / \varepsilon}^{+}} t^{1-2\sigma}|(x, t)|^3|\nabla W_{1}|^{2} \,\ud x \,\ud t<\infty
$$
if $n>2\sigma+3$. Therefore,
\be\label{eq:case2step2-1}
\int_{\mathcal{B}_{2 \delta}^{+} \backslash \mathcal{B}_{\delta}^{+}} t^{1-2\sigma}|\nabla W_{\varepsilon}|^{2} \,\ud x \,\ud t =\int_{\mathcal{B}_{2 \delta/\varepsilon}^{+} \backslash \mathcal{B}_{\delta/\varepsilon}^{+}} t^{1-2\sigma}|\nabla W_{1}|^{2} \,\ud x \,\ud t= O(\varepsilon^3).
\ee
In view of \eqref{eq:case1step3-1}, \eqref{eq:case1step3-2}, and \eqref{eq:case2step2-1}, we obtain
$$
I_2=\int_{\mathcal{B}_{2 \delta}^{+} \backslash \mathcal{B}_{\delta}^{+}} t^{1-2\sigma}|\nabla_g \phi_{\varepsilon}|^{2} \,\ud v_g = o(\varepsilon^{2}).
$$

{\bf Step 3:} Conclusion.

Using \eqref{eq:H>0-2} and Lemma \ref{lem:appendix1-2}, we have
$$
\kappa_{\sigma}^{\frac{n-2\sigma}{2\sigma}}S(n,\sigma)^{-\frac{n}{2\sigma}}=\int_{\mathbb{R}_{+}^{n+1}} t^{1-2 \sigma}|\nabla W_{1}|^{2} \,\ud x\,\ud t=\frac{(n-2)(n-2+2 \sigma)(n-2-2 \sigma)}{4 \sigma(n-1)} A_{0},
$$
i.e.,
\be\label{eq:A0}
A_{0}=\frac{4 \sigma(n-1)}{(n-2)(n-2+2 \sigma)(n-2-2 \sigma)}\kappa_{\sigma}^{\frac{n-2\sigma}{2\sigma}}S(n,\sigma)^{-\frac{n}{2\sigma}}.
\ee
Inserting \eqref{eq:A0} into \eqref{eq:case2I1}, we obtain
$$
\begin{aligned}
I_1\leq&\kappa_{\sigma}^{\frac{n-2\sigma}{2\sigma}}S(n,\sigma)^{-\frac{n}{2\sigma}}-\frac{n^{2}-4 n(1-\sigma)+4(1-\sigma-\sigma^{2})}{24 \sigma(n-1)}\bar{R}(P) A_{0}\varepsilon^2\\
&-\frac{(1-\sigma)(n-2-2\sigma)}{n}\|\pi\|^{2}(P)A_{0}\varepsilon^2\\
&-\frac{(1-\sigma)(3n-2-2\sigma)}{3n}R_{tt}(P)A_{0}\varepsilon^2 + o(\varepsilon^2)\\
=&\kappa_{\sigma}^{\frac{n-2\sigma}{2\sigma}}S(n,\sigma)^{-\frac{n}{2\sigma}}\Big(1-\frac{n^{2}-4 n(1-\sigma)+4(1-\sigma-\sigma^{2})}{6(n-2)(n-2+2 \sigma)(n-2-2 \sigma)}\bar{R}(P)\varepsilon^2\\
&-\frac{4\sigma(1-\sigma)(n-1)}{n(n-2)(n-2+2 \sigma)}\|\pi\|^{2}(P)\varepsilon^2\\
&-\frac{4\sigma(1-\sigma)(n-1)(3n-2-2\sigma)}{3n(n-2)(n-2+2 \sigma)(n-2-2 \sigma)}R_{tt}(P)\varepsilon^2+o(\varepsilon^2)\Big).
\end{aligned}
$$
Therefore,
$$
\begin{aligned}
&\int_{M}\rho^{1-2\sigma}|\nabla_g \phi_{\varepsilon}|^{2}\,\ud v_g\\
\leq&\kappa_{\sigma}^{\frac{n-2\sigma}{2\sigma}}S(n,\sigma)^{-\frac{n}{2\sigma}}\Big(1-\frac{n^{2}-4 n(1-\sigma)+4(1-\sigma-\sigma^{2})}{6(n-2)(n-2+2 \sigma)(n-2-2 \sigma)}\bar{R}(P)\varepsilon^2\\
&-\frac{4\sigma(1-\sigma)(n-1)}{n(n-2)(n-2+2 \sigma)}\|\pi\|^{2}(P)\varepsilon^2\\
&-\frac{4\sigma(1-\sigma)(n-1)(3n-2-2\sigma)}{3n(n-2)(n-2+2 \sigma)(n-2-2 \sigma)}R_{tt}(P)\varepsilon^2+o(\varepsilon^2)\Big)\\
=&:\kappa_{\sigma}^{\frac{n-2\sigma}{2\sigma}}S(n,\sigma)^{-\frac{n}{2\sigma}}(1+K_1\bar{R}(P)\varepsilon^2+K_2\|\pi\|^{2}(P)\varepsilon^2
+K_3R_{tt}(P)\varepsilon^2+o(\varepsilon^2)),
\end{aligned}
$$
where
$$
\begin{aligned}
K_1=&-\frac{n^{2}-4 n(1-\sigma)+4(1-\sigma-\sigma^{2})}{6(n-2)(n-2+2 \sigma)(n-2-2 \sigma)},\\
K_2=&-\frac{4\sigma(1-\sigma)(n-1)}{n(n-2)(n-2+2 \sigma)},\\
K_3=&-\frac{4\sigma(1-\sigma)(n-1)(3n-2-2\sigma)}{3n(n-2)(n-2+2 \sigma)(n-2-2 \sigma)}.
\end{aligned}
$$

Note that \eqref{eq:step1n>=5} is always true, hence
$$
\begin{aligned}
&\frac{\int_{M}\rho^{1-2\sigma}|\nabla_g \phi_{\varepsilon}|^{2}\,\ud v_g}{(\int_{\partial M} |\phi_{\varepsilon}|^{p}\,\ud s_g)^{2/p}}\\
=&S(n,\sigma)^{-1}(1+\bar{K}_1\bar{R}(P)\varepsilon^2+K_2\|\pi\|^{2}(P)\varepsilon^2
+K_3R_{tt}(P)\varepsilon^2+o(\varepsilon^2)),
\end{aligned}
$$
where
$$
\begin{aligned}
\bar{K}_1=&K_1+\frac{n-2\sigma}{6n(n-2)}\\
=&-\frac{\sigma(3n^2-6n-4\sigma^2+4)}{3n(n-2)(n-2+2 \sigma)(n-2-2 \sigma)}.
\end{aligned}
$$
By direct calculations,
$$
\bar{K}_1\bar{R}(P)+K_2\|\pi\|^{2}(P)+K_3R_{tt}(P)<0
$$
is equivalent to
$$
\frac{3n^2-6n-4\sigma^2+4}{12(1-\sigma)(n-1)(n-2-2\sigma)}\bar{R}(P)+\|\pi\|^2(P)+
\frac{3n-2-2\sigma}{3n-6-6\sigma}R_{tt}(P)>0.
$$
Under the hypothesis of this case, we obtain
$$
\frac{\int_{M}\rho^{1-2\sigma}|\nabla_g \phi_{\varepsilon}|^{2}\,\ud v_g}{(\int_{\partial M} |\phi_{\varepsilon}|^{p}\,\ud s_g)^{2/p}}<S(n,\sigma)^{-1}
$$
for $\varepsilon$ small enough. By Proposition \ref{pro:3}, this finishes the proof of Theorem \ref{thm}.

\appendix

\section{Appendix}

In this appendix, we will provide some lemmas used in the previous sections.
\begin{lemma}\label{lem:appendix1-1}
Let $\sigma\in (0,1)$, $n>2 \sigma$, and $W_{\varepsilon}=W_{\varepsilon, 0}$ is defined in \eqref{eq:Wvarsigma}. Then there holds
\begin{enumerate}[(1)]
  \item
$W_{\varepsilon}(x, t)=O(\varepsilon^{(n-2 \sigma) / 2}(\varepsilon^{2}+|(x,t)|^{2})^{-(n-2 \sigma) / 2})$,
  \item
$\nabla_{x} W_{\varepsilon}(x, t)=O(\varepsilon^{(n-2 \sigma) / 2}(\varepsilon^{2}+|(x,t)|^{2})^{-(n-2 \sigma+1) / 2})$,
  \item
$\partial_{t} W_{\varepsilon}(x, t)=O(\varepsilon^{(n-2 \sigma) / 2} t^{2 \sigma-1}(\varepsilon^{2}+|(x,t)|^{2})^{-n / 2})$.
\end{enumerate}
\end{lemma}

\begin{proof}
The proof can be found in \cite[Corollary 3.2]{MayerNdiayeFractional}, see also \cite[Lemma A.1]{NdiayeSireSunUniformization2021}.
\end{proof}

\begin{lemma}\label{lem:appendix1-2}
Let $\sigma\in (0,1)$, $n>2\sigma+2$, and $W_{1}=W_{1, 0}$ is defined in \eqref{eq:Wvarsigma}. Then there holds
\be\label{eq:appendix-3}
\int_{\mathbb{R}_{+}^{n+1}} t^{1-2 \sigma}|x|^{2}|\nabla W_{1}|^{2} \,\ud x\,\ud t=\frac{n(n^{2}-4 n(1-\sigma)+4(1-\sigma-\sigma^{2}))}{4 \sigma(n-1)} A_{0},
\ee
$$
\int_{\mathbb{R}_{+}^{n+1}} t^{1-2 \sigma}|\nabla W_{1}|^{2} \,\ud x\,\ud t=\frac{(n-2)(n-2+2 \sigma)(n-2-2 \sigma)}{4 \sigma(n-1)} A_{0},
$$
$$
\int_{\mathbb{R}_{+}^{n+1}} t^{3-2 \sigma}|\nabla_{x} W_{1}|^{2} \,\ud x\,\ud t=\frac{2(1-\sigma^{2})}{3} A_{0},
$$
$$
\int_{\mathbb{R}_{+}^{n+1}} t^{3-2 \sigma}|\nabla W_{1}|^{2} \,\ud x\,\ud t=2(1-\sigma) A_{0},
$$
where
$$
A_{0}=\int_{\mathbb{R}_{+}^{n+1}} t^{1-2 \sigma} W_{1}^{2} \,\ud x\,\ud t<\infty.
$$
\end{lemma}

\begin{proof}
We only prove \eqref{eq:appendix-3}, the others are similar. The idea of the proof is using the Plancherel theorem and an explicit formula of the Fourier transform of $W_1(x,t)$ in $x$, which is motivated from \cite{GonzalezQing2013}, as well as \cite{ChenDengKimClustered2017, GonzalezWang2018, KimMussoWei2017, KimMussoWei2018}.

Firstly, by the Plancherel theorem, we obtain
$$
\begin{aligned}
\int_{\mathbb{R}^{n}}|x|^{2}|\nabla_x W_1(x, t)|^{2} \,\ud x=&\sum_{i=1}^{n}\||\cdot|{ } \partial_i W_1(\cdot, t)\|_{L^{2}(\mathbb{R}^{n})}^{2}\\
=&\sum_{i=1}^{n} \int_{\mathbb{R}^{n}} \xi_{i} \widehat{W}_1(|\xi|, t) \cdot(-\Delta)_{\xi}(\xi_{i} \widehat{W}_1(|\xi|, t)) \,\ud \xi,
\end{aligned}
$$
where $\widehat{W}_1(\xi, t)$ is the Fourier transform of $W_1(x, t)$ with respect to the variable $x \in \mathbb{R}^{n}$. It follows from \cite{GonzalezGamma2009} and \cite{GonzalezWang2018} that
$$
\widehat{W}_1(\xi, t)=\hat{w}_1(\xi) \varphi(|\xi| t) \quad \text { for all }\, \xi \in \mathbb{R}^{n}\, \text { and }\, t>0,
$$
where $\hat{w}_1(r)$, $r=|\xi|$ satisfies
\be\label{eq:appendixequ-1}
\hat{w}_1^{\prime \prime}(r)+\frac{1+2 \sigma}{r} \hat{w}_1^{\prime}(r)-\hat{w}_1(r)=0,
\ee
and $\varphi(s)$ satisfies
\be\label{eq:appendixequ-2}
\varphi^{\prime \prime}(s)+\frac{1-2 \sigma}{s} \varphi^{\prime}(s)-\varphi(s)=0.
\ee
By some tedious calculations and the relation
$$
(-\Delta_{\xi})(\xi_{i} \widehat{W}_1)=-2 \partial_{i}\widehat{W}_1-\xi_{i}\Delta_{\xi} \widehat{W}_1
$$
and
$$
\Delta_{\xi} \widehat{W}_1=\widehat{W}_1^{\prime\prime}+(n-1)r^{-1}\widehat{W}_1^{\prime},
$$
where $'$ represents the differentiation with respect to the radial variable $|\xi|$, it holds that
$$
\begin{aligned}
&\int_{\mathbb{R}_{+}^{n+1}} t^{1-2 \sigma}|x|^{2}|\nabla_x W_{1}|^{2} \,\ud x\,\ud t\\
=&-(n-2\sigma)\omega_{n-1}\int_0^{\infty}\int_0^{\infty} t^{1-2 \sigma}\hat{w}_1(r) \varphi^2(r t)\hat{w}_1^{\prime}(r)r^n\,\ud r\,\ud t\\
&-(n+2\sigma)\omega_{n-1}\int_0^{\infty}\int_0^{\infty} t^{2-2 \sigma}\hat{w}_1^2(r) \varphi(r t)\varphi^{\prime}(r t)r^n\,\ud r\,\ud t\\
&-\omega_{n-1}\int_0^{\infty}\int_0^{\infty} t^{1-2 \sigma}(1+t^{2})\hat{w}_1^2(r) \varphi^2(r t)r^{n+1}\,\ud r\,\ud t\\
&-2\omega_{n-1}\int_0^{\infty}\int_0^{\infty} t^{2-2 \sigma}\hat{w}_1(r) \varphi(r t) \hat{w}_1^{\prime}(r) \varphi^{\prime}(r t)r^{n+1}\,\ud r\,\ud t\\
=&-(n-2\sigma)\omega_{n-1}\Big(\int_0^{\infty}r^{n-2+2\sigma}\hat{w}_1(r)\hat{w}_1^{\prime}(r)\,\ud r\Big)\Big(\int_0^{\infty}s^{1-2\sigma}\varphi^2(s)\,\ud s\Big)\\
&-(n+2\sigma)\omega_{n-1}\Big(\int_0^{\infty}r^{n-3+2\sigma}\hat{w}_1^2(r)\,\ud r\Big)\Big(\int_0^{\infty}s^{2-2\sigma}\varphi(s)\varphi^{\prime}(s)\,\ud s\Big)\\
&-\omega_{n-1}\Big(\int_0^{\infty}r^{n-1+2\sigma}\hat{w}_1^2(r)\,\ud r\Big)\Big(\int_0^{\infty}s^{1-2\sigma}\varphi^2(s)\,\ud s\Big)\\
&-\omega_{n-1}\Big(\int_0^{\infty}r^{n-3+2\sigma}\hat{w}_1^2(r)\,\ud r\Big)\Big(\int_0^{\infty}s^{3-2\sigma}\varphi^2(s)\,\ud s\Big)\\
&-2\omega_{n-1}\Big(\int_0^{\infty}r^{n-2+2\sigma}\hat{w}_1(r)\hat{w}_1^{\prime}(r)\,\ud r\Big)\Big(\int_0^{\infty}s^{2-2\sigma}\varphi(s)\varphi^{\prime}(s)\,\ud s\Big).
\end{aligned}
$$
Using \cite[Lemma B.2]{KimMussoWei2018}, we have
\be\label{eq:appendix-x}
\begin{aligned}
&\int_{\mathbb{R}_{+}^{n+1}} t^{1-2 \sigma}|x|^{2}|\nabla_x W_{1}|^{2} \,\ud x\,\ud t\\
=&\frac{(n+2)(3n^2-6n-4\sigma^2+4)}{8(n-1)(1-\sigma^2)}\omega_{n-1}\Big(\int_0^{\infty}r^{n-3+2\sigma}\hat{w}_1^2(r)\,\ud r\Big)\Big(\int_0^{\infty}s^{3-2\sigma}\varphi^2(s)\,\ud s\Big).
\end{aligned}
\ee

Next, we will calculate $\int_{\mathbb{R}^{n}}|x|^{2}|\partial_t W_1(x, t)|^{2} \,\ud x$. By the Plancherel theorem,
$$
\begin{aligned}
\int_{\mathbb{R}^{n}}|x|^{2}|\partial_t W_1(x, t)|^{2} \,\ud x=&\||\cdot|{ } \partial_t W_1(\cdot, t)\|_{L^{2}(\mathbb{R}^{n})}^{2}\\
=&\int_{\mathbb{R}^{n}}|\xi| \hat{w}_1(\xi) \varphi^{\prime}(|\xi| t) \cdot(-\Delta)_{\xi}(|\xi| \hat{w}_1(\xi) \varphi^{\prime}(|\xi| t)) \,\ud \xi.
\end{aligned}
$$
Then employing \eqref{eq:appendixequ-1} and \eqref{eq:appendixequ-2}, we obtain
$$
\begin{aligned}
&\Delta_{\xi}(|\xi| \hat{w}_1(\xi) \varphi^{\prime}(|\xi| t))\\
=&(r \hat{w}_1(r) \varphi^{\prime}(r t))^{\prime\prime}+\frac{n-1}{r}(r \hat{w}_1(r) \varphi^{\prime}(r t))^{\prime}\\
=&(n-2+2\sigma)\hat{w}_1^{\prime}(r)\varphi^{\prime}(r t)+2rt\hat{w}_1^{\prime}(r)\varphi(r t)\\
&+\Big(\frac{2\sigma(n-2+2\sigma)}{r}+r(1+t^2)\Big)\hat{w}_1(r)\varphi^{\prime}(r t)+(n+2\sigma)t\hat{w}_1(r)\varphi(r t).
\end{aligned}
$$
Therefore, using \cite[Lemma B.2]{KimMussoWei2018},
\be\label{eq:appendix-t}
\begin{aligned}
&\int_{\mathbb{R}_{+}^{n+1}} t^{1-2 \sigma}|x|^{2}|\partial_t W_{1}|^{2} \,\ud x\,\ud t\\
=&-(n-2+2\sigma)\omega_{n-1}\Big(\int_{0}^{\infty}r^{n-2+2\sigma}\hat{w}_1(r)\hat{w}_1^{\prime}(r) \,\ud r\Big) \Big(\int_{0}^{\infty}s^{1-2 \sigma}\varphi^{\prime 2}(s)\,\ud s\Big)\\
&-2\omega_{n-1}\Big(\int_{0}^{\infty}r^{n-2+2\sigma}\hat{w}_1(r)\hat{w}_1^{\prime}(r) \,\ud r\Big) \Big(\int_{0}^{\infty}s^{2-2 \sigma}\varphi(s)\varphi^{\prime}(s)\,\ud s\Big)\\
&-2\sigma(n-2+2\sigma)\omega_{n-1}\Big(\int_{0}^{\infty}r^{n-3+2\sigma}\hat{w}_1^2(r) \,\ud r\Big) \Big(\int_{0}^{\infty}s^{1-2 \sigma}\varphi^{\prime 2}(s)\,\ud s\Big)\\
&-\omega_{n-1}\Big(\int_{0}^{\infty}r^{n-1+2\sigma}\hat{w}_1^2(r) \,\ud r\Big) \Big(\int_{0}^{\infty}s^{1-2 \sigma}\varphi^{\prime 2}(s)\,\ud s\Big)\\
&-\omega_{n-1}\Big(\int_{0}^{\infty}r^{n-3+2\sigma}\hat{w}_1^2(r) \,\ud r\Big) \Big(\int_{0}^{\infty}s^{3-2 \sigma}\varphi^{\prime 2}(s)\,\ud s\Big)\\
&-(n+2\sigma)\omega_{n-1}\Big(\int_{0}^{\infty}r^{n-3+2\sigma}\hat{w}_1^2(r) \,\ud r\Big) \Big(\int_{0}^{\infty}s^{2-2 \sigma}\varphi(s)\varphi^{\prime}(s)\,\ud s\Big)\\
=&\frac{3n^3-12n^2-4n\sigma^2+8n\sigma+12n-8\sigma^2-8\sigma}{8\sigma(1+\sigma)(n-1)}\omega_{n-1}\Big(\int_0^{\infty}r^{n-3+2\sigma}\hat{w}_1^2(r)\,\ud r\Big)\times\\
&\Big(\int_0^{\infty}s^{3-2\sigma}\varphi^2(s)\,\ud s\Big).
\end{aligned}
\ee

Similarly,
\be\label{eq:appendix-A0}
\begin{aligned}
A_{0}=&\int_{\mathbb{R}_{+}^{n+1}} t^{1-2 \sigma} W_{1}^{2} \,\ud x\,\ud t\\
=&\frac{3}{2(1-\sigma^2)}\omega_{n-1}\Big(\int_0^{\infty}r^{n-3+2\sigma}\hat{w}_1^2(r)\,\ud r\Big)\Big(\int_0^{\infty}s^{3-2\sigma}\varphi^2(s)\,\ud s\Big).
\end{aligned}
\ee

Combining \eqref{eq:appendix-x}, \eqref{eq:appendix-t} and \eqref{eq:appendix-A0} together, we see that,
$$
\int_{\mathbb{R}_{+}^{n+1}} t^{1-2 \sigma}|x|^{2}|\nabla W_{1}|^{2} \,\ud x\,\ud t=\frac{n(n^2-4n(1-\sigma)+4(1-\sigma-\sigma^2))}{4\sigma(n-1)}A_0.
$$
This finishes the proof.
\end{proof}

\bigskip

\noindent Z. Tang \& N. Zhou

\noindent School of Mathematical Sciences, Laboratory of Mathematics and Complex Systems, MOE, \\
Beijing Normal University, Beijing, 100875, China\\[1mm]
Email: \textsf{tangzw@bnu.edu.cn}\\
Email: \textsf{nzhou@mail.bnu.edu.cn}

\end{document}